\documentclass[letterpaper]{amsart}
\usepackage[utf8]{inputenc}
\usepackage[T1]{fontenc}
\usepackage{xspace}
\usepackage{amsmath,amsthm,amssymb}
\usepackage{ifthen}
\usepackage[all]{xy}
\usepackage{graphicx}

\newtheorem{theorem}{Theorem}%[section]
\newtheorem{corollary}[theorem]{Corollary}
\newtheorem{lemma}[theorem]{Lemma}
\newtheorem{proposition}[theorem]{Proposition}

\newtheorem{question}[theorem]{Question}

\theoremstyle{definition}
\newtheorem{definition}[theorem]{Definition}

\renewcommand{\P}{\mathbb{P}}
\newcommand{\Q}{\mathbb{Q}}
\newcommand{\R}{\mathbb{R}}
\newcommand{\C}{\mathbb{C}}
\newcommand{\B}{\mathbb{B}}
\newcommand{\D}{\mathcal{D}}
\renewcommand{\c}{\mathfrak{c}}
\newcommand{\grma}{\ensuremath{\mathrm{grMA}}\xspace}
\newcommand{\grpfa}{\ensuremath{\mathrm{grPFA}}\xspace}
\newcommand{\ma}{\ensuremath{\mathrm{MA}}\xspace}
\newcommand{\mac}{\ensuremath{\mathrm{MA}(\textup{countable})}\xspace}
\newcommand{\macoh}{\ensuremath{\mathrm{MA}(\textup{Cohen})}\xspace}
\newcommand{\mas}{\ensuremath{\mathrm{MA}(\textup{\(\sigma\)-centred})}\xspace}
\newcommand{\mak}{\ensuremath{\mathrm{MA}(\textup{Knaster})}\xspace}
\newcommand{\pfa}{\ensuremath{\mathrm{PFA}}\xspace}
\DeclareMathOperator{\Term}{Term}
\newcommand{\Termfin}{\Term_{\mathrm{fin}}}
\newcommand{\Bnull}{\B_{\textup{null}}}
\DeclareMathOperator{\cf}{cf}
\newcommand{\set}[2]{\{\,#1\,;\,#2\,\}}

\DeclareMathOperator{\Add}{Add}
\newcommand{\Ptail}{\P_{\mathrm{tail}}}

\newcommand{\Htail}{H_{\mathrm{tail}}}
\DeclareMathOperator{\Coll}{Coll}
\DeclareMathOperator{\dom}{dom}

\newcommand{\leftexp}[2]{{\vphantom{#2}}^{#1}{#2}}
\newcommand{\funcs}[2]{\leftexp{#1}{#2}}
\newcommand{\forces}{\Vdash}
\newcommand{\rest}{\mathbin{\upharpoonright}}

\newcommand*{\concat}{\mathchoice{\mathbin{\raisebox{0.7ex}{$\smallfrown$}}}{\mathbin{\raisebox{0.7ex}{$\smallfrown$}}}{\mathbin{\raisebox{0.15ex}{$\smallfrown$}}}{\mathbin{\raisebox{0.15ex}{$\smallfrown$}}}}

\ifthenelse{\isundefined{\shownotes}}{
  \newcommand{\note}[1]{}
}{
  \newcommand{\note}[1]{\marginpar{\raggedright\footnotesize #1}}
}

\begin{document}
\title{The grounded Martin's Axiom}
\author{Miha E.\ Habič}
\address{M.\ Habic\\ The Graduate Center of the City University of New York\\ Mathematics\\
365 Fifth Avenue\\New York\\ NY 10016\\ USA}
\email{mhabic@gradcenter.cuny.edu}
\begin{abstract}
We introduce a variant of Martin's Axiom, called the grounded Martin's Axiom, or \grma, which
asserts that the universe is a ccc forcing extension in which Martin's
Axiom holds for posets in the ground model. This principle already implies
several of the combinatorial consequences of \ma.
The new axiom is shown to be consistent 
with the failure of \ma and a singular continuum. We prove that 
\grma is preserved in a 
strong way when adding a Cohen real and that adding a random real to a model of \ma 
preserves \grma (even though it destroys \ma itself). We also consider the
analogous variant of the Proper Forcing Axiom.
\end{abstract}

\maketitle
The standard Solovay-Tennenbaum proof of the consistency of Martin's Axiom with
a large continuum starts by choosing a suitable cardinal \(\kappa\) and then
proceeds in an iteration of length \(\kappa\) by forcing with ccc posets of size
less than \(\kappa\), and not just those in the ground model but also those
arising in the intermediate extensions. To ensure that all of the potential ccc posets 
are considered, some bookkeeping device is usually employed.

Consider now the following reorganization of the argument. Instead of iterating
for \(\kappa\) many steps we build a length \(\kappa^2\) iteration by first dealing
with the \(\kappa\) many small posets in the ground model, then the small posets in
that extension, and so on. The full length \(\kappa^2\) iteration can be seen as a length 
\(\kappa\) iteration, all of whose iterands are themselves length \(\kappa\) iterations.

In view of this reformulation, we can ask what happens if we halt this
construction after forcing with the first length \(\kappa\) iteration, when we have, in 
effect, ensured that Martin's Axiom holds for posets from the ground model. What 
combinatorial consequences of Martin's Axiom follow already from this weaker principle? 
We aim in this paper to answer these questions (at least partially).

\section{Forcing the grounded Martin's Axiom}

\begin{definition}
%The \emph{grounded Martin's Axiom} (\grma) asserts that \(V\) is a ccc forcing
%extension of some ground model \(W\) and if \(\Q\in W\) is a poset which is still
%ccc in \(V\) and \(\mathcal{D}\in V\) is a family of fewer than \(\c^V\) many
%dense subsets of \(\Q\) then there is in \(V\) a \(\mathcal{D}\)-generic filter
%on \(\Q\).
%
The \emph{grounded Martin's Axiom} (\grma) asserts that \(V\) is a ccc forcing
extension of some ground model \(W\) and \(V\) satisfies the conclusion of
Martin's Axiom for posets \(\Q\in W\) which are still ccc in \(V\).
\end{definition}

To be clear, in the definition we require that \(V\) have \(\mathcal{D}\)-generic
filters for any \(\mathcal{D}\in V\) a family of fewer than \(\c^V\) dense
subsets of \(\Q\).
We should also note that, while the given definition is second-order, \grma is in fact 
first-order expressible, using the result of Reitz \cite{Reitz2007:GroundAxiom} that the 
ground models of the universe are uniformly definable.

%Since we are mainly interested in the combinatorial consequences of \grma, the above definition
%isn't quite satisfactory, since it is possible to prevent the universe being a set-forcing
%extension by arbitrarily closed class forcing. A slightly more technical definition, which
%avoids this problem, suggests itself.
%
%\begin{definition}
%\note{Does this make sense? Some \(\kappa\) or all \(\kappa\)?}
%\grma is the assertion that there exists a cardinal \(\kappa\geq\c\) such that \(H_{\kappa^+}\)
%is a ccc forcing extension of some model \(W\) of \(\mathrm{ZFC}^-\) and that for any poset 
%\(\Q\in W\) of size less than \(\c^V\) that is ccc in \(V\) and any family \(\mathcal{D}\in V\)
%of less than \(\c^V\) many dense subsets of \(\Q\) there exists in \(V\) a 
%\(\mathcal{D}\)-generic filter on \(\Q\).
%\end{definition}
%
%If \(V\) satisfies the first version of \grma, it also satisfies the second. We will mainly
%confine ourselves to dealing with the first version.

If Martin's Axiom holds, we may simply take \(W=V\) in the definition,
which shows that the grounded Martin's Axiom is implied by Martin's Axiom.
In particular, by simply performing the usual Martin's Axiom iteration,
\(\grma+\c=\kappa\) can be forced from any model where \(\kappa\) is regular
and satisfies \(2^{<\kappa}=\kappa\). It will be shown in theorem~\ref{thm:CanonicalgrMA},
however, that the grounded Martin's Axiom is strictly weaker than Martin's Axiom.
Ultimately we shall see that \grma retains some of the interesting combinatorial
consequences of Martin's Axiom (corollary~\ref{cor:grMALargeCharacts}), while 
also being more robust with respect to mild forcing 
(theorems~\ref{thm:CohenPreservesgrMA} and \ref{thm:RandomGivesgrMA}).

As in the case of Martin's Axiom, a key property of the grounded Martin's Axiom is
that it is equivalent to its restriction to posets of small size.

\begin{lemma}
\label{lemma:grMASmallPosets}
The grounded Martin's Axiom is equivalent to its restriction to posets of size less 
than continuum, i.e.\ the following principle:

{\newlength{\bit}%
\setlength{\bit}{\linewidth}%
\addtolength{\bit}{-5em}%
\vspace{\baselineskip}%
%\textup{(\(\ast\))}\hspace{2em}%
\parbox{\bit}{The universe \(V\) is a ccc forcing extension of some ground
model \(W\) and \(V\) satisfies the conclusion of Martin's Axiom for
posets \(\Q\in W\) of size less than \(\c^V\) which are still ccc in \(V\).}%
\hspace*{2em}\textup{(\(\ast\))}%
\vspace*{\baselineskip}}
\end{lemma}

\begin{proof}
Assume \(V\) satisfies (\(\ast\)) and let \(\Q\in W\) be a poset which is ccc in \(V\) and 
\(\D=\set{D_\alpha}{\alpha<\kappa}\in V\)
a family of \(\kappa<\c^V\) many dense subsets of \(\Q\). Let \(V=W[G]\) for some
\(W\)-generic \(G\subseteq \P\) and let \(\dot{D}_\alpha\in W\) be \(\P\)-names for the
\(D_\alpha\). Choose \(\theta\) large enough so that \(\P,\Q\) and all of the 
\(\dot{D}_\alpha\) are in \(H_{\theta^+}^W\).
We can then find an \(X\in W\) of size at most \(\kappa\) such that 
\(X\prec H_{\theta^+}^W\) and \(X\) contains \(\P,\Q\) and the 
\(\dot{D}_\alpha\).

Now let \[X[G]=\set{\tau^G}{\text{\(\tau \in X\) is a \(\P\)-name}}\in W[G]\]
We can verify the Tarski-Vaught criterion to show that \(X[G]\) is an
elementary substructure of \(H_{\theta^+}^W[G]=H_{\theta^+}^{W[G]}\). Specifically,
suppose that \(H_{\theta^+}^W[G]\models\exists x\colon\varphi(x,\tau^G)\)
for some \(\tau\in X\). Let \(S\) be the set of conditions \(p\in\P\) which force
\(\exists x\colon\varphi(x,\tau)\). Since \(S\) is definable from the parameters
\(\P\) and \(\tau\) we get \(S\in X\). Let \(A\in X\) be a antichain, maximal
among those contained in \(S\). By mixing over \(A\) we can obtain a name
\(\sigma\in X\) such that \(p\forces \varphi(\sigma,\tau)\) for
any \(p\in A\) and it follows that
\(H_{\theta^+}^W[G]\models\varphi(\sigma^G,\tau^G)\), which completes the verification.

Now let \(\Q^*=\Q\cap X[G]\) and \(D_\alpha^*=D_\alpha\cap X[G]\). Then
\(\Q^*\prec \Q\) and it follows that \(\Q^*\) is ccc (in \(W[G]\)) and that 
\(D_\alpha^*\) is dense in \(\Q^*\) for any \(\alpha<\kappa\). Furthermore, \(\Q^*\) has
size at most \(\kappa\). Finally, since \(\P\) is ccc,
the filter \(G\) is \(X\)-generic and so \(\Q^*=\Q\cap X[G]=\Q\cap X\) is an element
of \(W\). If we now apply \((\ast)\) to \(\Q^*\) and \(\mathcal{D}^*=\set{D_\alpha^*}{\alpha<\kappa}\), we find in \(W[G]\) a filter \(H\subseteq \Q^*\) intersecting every \(D_\alpha^*\), and
thus every \(D_\alpha\). Thus \(H\) generates a \(\mathcal{D}\)-generic filter on \(\Q\).
\end{proof}

%Since \(\P\) is ccc, the filter \(G\) is \(X\)-generic for \(\P\). Thus, if we
%define \[X[G]=\set{\tau^G}{\text{\(\tau \in X\) is a \(\P\)-name}}\in W[G]\]
%we can conclude that \(x\cap X=x\cap X[G]\) for any \(x\in X\) by the usual genericity arguments.
%
%Consider the poset \(\Q^*=\Q\cap X\in W\); it has size at most \(\kappa\).
%The conclusion in the second to last paragraph shows that \(\Q^*=\Q\cap X[G]\), 
%which implies that \(\Q^*\) is ccc in \(W[G]\). To see this, let \(A\in W[G]\) be an
%uncountable subset of \(\Q^*\). Since \(\Q\) is ccc in \(W[G]\) by assumption, 
%\(A\) cannot be an antichain of \(\Q\): there are some \(q_1,q_2\in A\) which are 
%compatible in \(\Q\). But then, by elementarity, \(X\) also has a lower bound for 
%\(q_1\) and \(q_2\), therefore \(A\) is not an antichain in \(\Q^*\) either.
% 
%Since we arranged that \(X\) contains their names, all of the \(D_\alpha\) are in \(X[G]\).
%Let \(D_\alpha^*=D_\alpha\cap X[G]\in W[G]\); by elementarity these are dense subsets of 
%\(\Q^*\). Applying (\(\ast\)) to \(\Q^*\) and \(\mathcal{D}^*=\set{D_\alpha^*}{\alpha<\kappa}\)
%we find in 

The reader has likely noticed that the proof of 
lemma~\ref{lemma:grMASmallPosets} is somewhat more involved than the proof of
the analogous result for Martin's Axiom. The argument there hinges on the
straightforward observation that elementary subposets of ccc posets are themselves
ccc. While that remains true in our setting, of course, matters are made more difficult 
since we require that all of our posets come from a ground model that may not
contain the dense sets under consideration.
It is therefore not at all clear that taking
appropriate elementary subposets will land us in the ground model and a slightly more
elaborate argument is needed.

Let us point out a deficiency in the definition of \grma. As we have described 
it, the principle posits the existence of a ground model for the universe, a 
considerable global assumption. On the other hand, lemma~\ref{lemma:grMASmallPosets} 
suggests that the 
operative part of the axiom is, much like Martin's Axiom, a statement about \(H_\c\). 
This discrepancy allows for some undesirable phenomena. For example, 
Reitz~\cite{Reitz2007:GroundAxiom} shows that it is possible to perform arbitrarily 
closed class forcing over a given model and obtain a model of the ground axiom,
the assertion that the universe is not a nontrivial set forcing extension over any
ground model at all. This 
implies that there are models which have the same \(H_\c\) as a model of the
grounded Martin's Axiom but which fail to satisfy it simply because they have
no nontrivial ground models at all. 
To avoid this situation we can weaken the definition of \grma in a technical way.

\begin{definition}
The \emph{local grounded Martin's Axiom} asserts that there are a cardinal
\(\kappa\geq\c\) and a transitive \(\mathrm{ZFC}^-\) model 
\(M\subseteq H_{\kappa^+}\) such that
\(H_{\kappa^+}\) is a ccc forcing extension of \(M\) and \(V\) satisfies the
conclusion of Martin's Axiom for posets \(\Q\in M\) which are still ccc in \(V\).
\end{definition}

Of course, if the grounded Martin's Axiom holds, over the ground model \(W\) via
the forcing notion \(\P\), then its 
local version holds as well. We can simply take \(\kappa\) to be large enough
so that \(M=H_{\kappa^+}^W\) contains \(\P\) and that \(M[G]=H_{\kappa^+}^V\).
One should view the local version of the axiom as capturing all of the relevant
combinatorial effects of \grma (which, as we have seen, only involve \(H_\c\)),
while disentangling it from the structure of the universe higher up.
%The reader can easily verify that all of the results in this paper really only 
%involve the bounded, combinatorial content of \grma and so we could have replaced
%our use of the axiom everywhere by its local version while preserving the salient
%points of each proof.

We now aim to give a model where the Martin's Axiom fails but the grounded version
holds. The idea is to imitate the Solovay-Tennenbaum argument, but to only use ground
model posets in the iteration. While it is then relatively clear that \grma will hold
in the extension, a further argument is needed to see that \ma itself fails. 
The key will be a kind of product analysis, given in the next few lemmas. We will show
that the iteration of ground model posets, while not exactly
a product, is close enough to a product to prevent Martin's Axiom from holding in
the final extension by a result of Roitman. An extended version of this argument will
also yield the consistency of \grma with a singular continuum.

\begin{lemma}
\label{lemma:TwoStepIterationDecidesccc}
Let \(\Q_0\) and \(\R\) be posets and \(\tau\) a \(\Q_0\)-name for a poset
such that \(\Q_0*\tau\) is ccc and
\(\Q_0\forces\textup{``if \(\check{\R}\) is ccc then \(\tau=\check{\R}\)''}\).
Furthermore, suppose that \(\Q_0*\tau\forces\textup{``\(\check{\Q}_0\) is ccc''}\). Then either
\(\Q_0\forces\textup{``\(\check{\R}\) is ccc''}\) or \(\Q_0\forces\textup{``\(\check{\R}\) is not ccc''}\).
\end{lemma}

\begin{proof}
Suppose the conclusion fails, so that there are conditions \(q_0,q_1\in\Q_0\) which
force \(\check{\R}\) to either be or not be ccc, respectively. It follows that
\(\Q_0\rest q_1\times\R\) is not ccc. Switching the factors, there must be a condition
\(r\in\R\) forcing that \(\check{\Q}_0\rest\check{q}_1\) is not ccc. Now let \(G*H\)
be generic for \(\Q_0*\tau\) with \((q_0,\check{r})\in G*H\) (note that 
\((q_0,\check{r})\) is really a condition since \(q_0\forces\tau=\check{\R}\)).
Consider the extension \(V[G*H]\). On the one hand \(\Q_0\) must be ccc there, since
this was one of the hypotheses of our statement, but on the other hand \(\Q_0\rest q_1\) 
is not ccc there since \(r\in H\) forces this.
\end{proof}

\begin{lemma}
\label{lemma:FSIterationFactorsAsProduct}
Let \(\P=\langle \P_\alpha,\tau_\alpha;\alpha<\gamma\rangle\), with \(\gamma>0\), be a finite-support ccc
iteration such that for each \(\alpha\) there is some poset \(\Q_\alpha\) for which
\[
\P_\alpha\forces\textup{``if \(\check{\Q}_\alpha\) is ccc then \(\tau_\alpha=\check{\Q}_\alpha\) and \(\tau_\alpha\) is trivial otherwise''}
\]
Furthermore assume that \(\P\forces\textup{``\(\check{\Q}_0\) is ccc''}\). Then
\(\P\) is forcing equivalent to the product \(\Q_0\times\overline{\P}\) for some
poset \(\overline{\P}\).
\end{lemma}

Before we give the (technical) proof, let us provide some intuition for this lemma.
We can define the iteration \(\overline{\P}\) in the same way as \(\P\) (i.e.\ using the
same \(\Q_\alpha\)) but skipping the first step of forcing.
The idea is that, by lemma~\ref{lemma:TwoStepIterationDecidesccc}, the posets which
appear in the iteration \(\P\) do not depend on the first stage of forcing \(\Q_0\).
We thus expect that generics \(G\subseteq\P\) will correspond exactly to generics
\(H\times\overline{G}\subseteq \Q_0\times\overline{\P}\), since the first stage \(G_0\)
of the generic \(G\) does not affect the choice of posets in the rest of
the iteration.
 
\begin{proof}
We show the lemma by induction on \(\gamma\), the length of the iteration \(\P\).
In fact, we shall work with a stronger induction hypothesis. Specifically, we shall show
that for each \(\alpha<\gamma\) there is a poset \(\overline{\P}_\alpha\) such that
\(\Q_0\times\overline{\P}_\alpha\) embeds densely into \(\P_\alpha\), that
the \(\overline{\P}_\alpha\) form the initial segments of a finite-support iteration
%of length \(\gamma-1\) 
and that the dense embeddings extend one another. For the
purposes of this proof we shall take all two-step iterations to be in the style of
Kunen, i.e.\ the conditions in \(\P*\tau\) are pairs \((p,\sigma)\) such that
\(p\in\P\) and \(\sigma\in\dom(\tau)\) and \(p\forces\sigma\in\tau\). Furthermore
we shall assume that the \(\tau_\alpha\) are full names. These assumptions make no
difference for the statement of the lemma, but ensure that certain embeddings will in
fact be dense.

Let us start with the base case \(\gamma=2\), when \(\P=\Q_0*\tau_1\). 
By lemma~\ref{lemma:TwoStepIterationDecidesccc} whether or not \(\check{\Q}_1\) is ccc
is decided by every condition in \(\Q_0\). But then, by our assumption on \(\tau_1\),
if \(\Q_0\) forces \(\check{\Q}_1\) to be ccc then \(\tau_1\) is forced to be equal
to \(\check{\Q}_1\) and \(\Q_0\times\Q_1\) embeds densely into \(\P\), and otherwise
\(\tau_1\) is forced to be trivial and \(\Q_0\) embeds densely into \(\P\). Depending
on which is the case, we can thus take either \(\overline{\P}_2=\Q_1\) or 
\(\overline{\P}_2=1\).

For the induction step let us assume that the stronger induction hypothesis holds for 
iterations of length \(\gamma\) and show that it holds for iterations of length 
\(\gamma+1\).
Let us write \(\P=\P_\gamma*\tau_\gamma\). By the induction hypothesis there is
a \(\overline{\P}_\gamma\) such that \(\Q_0\times\overline{\P}_\gamma\) embeds densely
into \(\P_\gamma\).

Before we give the details, let us sketch the string of equivalences that will yield
the desired conclusion. We have
\[
\P\equiv (\Q_0\times\overline{\P}_\gamma)*\bar{\tau}_\gamma \equiv
(\overline{\P}_\gamma\times\Q_0)*\bar{\tau}_\gamma \equiv
\overline{\P}_\gamma*(\check{\Q}_0*\bar{\tau}_\gamma)
\]
Here \(\bar{\tau}_\gamma\) is the \(\Q_0\times\overline{\P}_\gamma\)-name (or
\(\overline{\P}_\gamma\times\Q_0\)-name) resulting from pulling back the 
\(\P_\gamma\)-name \(\tau_\gamma\) along the dense embedding provided by the induction
hypothesis. We will specify what exactly we mean by 
\(\check{\Q}_0*\bar{\tau}_\gamma\) later.

We can apply the base step of the induction to the iteration 
\(\check{\Q}_0*\bar{\tau}_\gamma\)
in \(V^{\overline{\P}_\gamma}\) and obtain a \(\overline{\P}_\gamma\)-name
\(\dot{\R}\) such that \(\check{\Q}_0\times\dot{\R}\) is forced to densely embed into
\(\check{\Q}_0*\bar{\tau}_\gamma\). We can then continue the chain above with
\[
\overline{\P}_\gamma*(\check{\Q}_0*\bar{\tau}_\gamma) \equiv
\overline{\P}_\gamma*(\check{\Q}_0\times\dot{\R}) \equiv
\overline{\P}_\gamma*(\dot{\R}\times\check{\Q}_0) \equiv
\overline{\P}_{\gamma+1}\times\Q_0
\]
where \(\overline{\P}_{\gamma+1}=\overline{\P}_\gamma*\dot{\R}\). While this is 
apparently enough to finish the successor step for the bare statement of the lemma,
we wish to preserve the stronger induction hypothesis, and this requires a bit more work.

We first pick a specific \(\overline{\P}_\gamma\)-name for \(\check{\Q}_0*\bar{\tau}_\gamma\).
Let
\[
\tau=\set{((\check{q}_0,\rho),\bar{p})}{q_0\in\Q_0, \rho\in\dom(\bar{\tau}_\gamma),\bar{p}\in\overline{\P}_\gamma,(\bar{p},q_0)\forces \rho\in\bar{\tau}_\gamma}
\]
Then \(\overline{\P}_\gamma\forces \tau=\check{\Q}_0*\tau_\gamma\). Next we pin down
\(\dot{\R}\) more. Note that \(\overline{\P}_\gamma\) forces that \(\check{\Q}_0\)
decides whether \(\check{\Q}_\gamma\) is ccc or not by 
lemma~\ref{lemma:TwoStepIterationDecidesccc}. Let 
\(A=A_0\cup A_1\subseteq\overline{\P}_\gamma\) be a maximal antichain such that each
\(\bar{p}\in A_0\) forces \(\check{\Q}_\gamma\) to not be ccc and each \(\bar{p}\in A_1\)
forces it to be ccc. Now let
\[
\dot{\R}=\set{(\check{1},\bar{p})}{\bar{p}\in A_1}\cup \set{(\check{q},\bar{p})}{q\in\Q_\gamma,\bar{p}\in A_0}
\]
Observe that \(\dot{\R}\) has the properties we require of it: it is forced by 
\(\overline{\P}_\gamma\) that \(\dot{\R}=\check{\Q}_\gamma\) if \(\check{\Q}_0\) forces
that \(\check{\Q}_\gamma\) is ccc, and \(\dot{\R}\) is trivial otherwise, and that
\(\check{\Q}_0\times \dot{\R}\) embeds densely into \(\tau\).

Finally, let us define \(\overline{\P}_{\gamma+1}=\overline{\P}_\gamma*\dot{\R}\).
We can now augment the equivalences given above with dense embeddings:

\begin{itemize}
\item The embedding \(\overline{\P}_{\gamma+1}\times\Q_0\hookrightarrow
\overline{\P}_\gamma*(\check{\Q}_0\times\dot{\R})\) is clear.

\item To embed \(\overline{\P}_\gamma*(\check{\Q}_0\times\dot{\R})\) into 
\(\overline{\P}_\gamma*\tau\) we can send \((\bar{p},(\check{q}_0,\rho))\) to \((\bar{p},
(\check{q}_0,\rho'))\) where \(\rho'\) is some element of \(\dom(\bar{\tau}_\gamma)\) for which
\((\bar{p},q_0)\forces \rho=\rho'\) (this is where the fullness of \(\tau_\gamma\) is
needed).

\item With our specific choice of the name \(\tau\) we in fact get an isomorphism
between \(\overline{\P}_\gamma*\tau\) and 
\((\overline{\P}_\gamma\times\Q_0)*\bar{\tau}_\gamma\), given by sending
\((\bar{p},(\check{q}_0,\rho))\) to \(((\bar{p},q_0),\rho)\).

\item The final embedding from \((\overline{\P}_\gamma\times\Q_0)*\bar{\tau}_\gamma\) into 
\(\P\) is given by the induction hypothesis.
\end{itemize}

After composing these embeddings, we notice that the first three steps essentially fixed 
the \(\overline{\P}_\gamma\) part of the condition and the last step fixed the
\(\tau_\gamma\) part. It follows that the embedding 
\(\overline{\P}_{\gamma+1}\times\Q_0\hookrightarrow \P_{\gamma+1}\) we constructed
extends the embedding
\(\overline{\P}_\gamma\times\Q_0\hookrightarrow \P_\gamma\) given by the induction
hypothesis. This completes the successor step of the induction.

We now look at the limit step of the induction. The induction hypothesis gives us
for each \(\alpha<\gamma\) a poset \(\overline{\P}_\alpha\) and a dense embedding
\(\overline{\P}_\alpha\times\Q_0\hookrightarrow \P_\alpha\) and we also know that
the \(\overline{\P}_\alpha\) are the initial segments of a finite-support iteration
and that the dense embeddings extend each other. If we now let \(\overline{\P}_\gamma\)
be the direct limit of the \(\overline{\P}_\alpha\), we can easily find a dense embedding
of \(\overline{\P}_\gamma\times\Q_0\) into \(\P_\gamma\). Specifically, given a condition
\((\bar{p},q)\in\overline{\P}_\gamma\times\Q_0\), we can find an \(\alpha<\gamma\)
such that \(\bar{p}\) is essentially a condition in \(\overline{\P}_\alpha\), since
\(\overline{\P}_\gamma\) is the direct limit of these. Now we can map
\((\bar{p},q)\) using the stage \(\alpha\) dense embedding, landing in
\(\P_\alpha\) and interpreting this as an element of \(\P_\gamma\). This map
is independent of the particular choice of \(\alpha\) since all the dense embeddings
extend one another and it is itself a dense embeddings since all the previous stages
were.
\end{proof}

\begin{theorem}
\label{thm:CanonicalgrMA}
Let \(\kappa>\omega_1\) be a cardinal of uncountable cofinality satisfying 
\(2^{<\kappa}=\kappa\). Then there is a ccc forcing extension that satisfies 
\(\grma + \lnot\ma + \c=\kappa\).
\end{theorem}

\begin{proof}
Fix a well-order \(\triangleleft\) of \(H_\kappa\) of length \(\kappa\); 
this can be done since
\(2^{<\kappa}=\kappa\). We can assume, without loss of generality, that the least element
of this order is the poset \(\Add(\omega,1)\). 
We define a length \(\kappa\) finite-support iteration
\(\P\) recursively: at stage \(\alpha\) we shall force with the next poset with respect 
to the order \(\triangleleft\) if that is ccc at that stage and with
trivial forcing otherwise. Let \(G\) be \(\P\)-generic.

Notice that any poset \(\Q\in H_\kappa\) occurs, up to isomorphism, unboundedly often in 
the well-order \(\triangleleft\). Specifically, we can first find an isomorphic copy 
whose universe is a set of ordinals bounded
in \(\kappa\) and then simply move this universe higher and higher up. In particular,
isomorphic copies of Cohen forcing \(\mathrm{Add}(\omega,1)\) appear unboundedly
often. Since these are ccc in a highly robust way (being
countable), they will definitely be forced with in the iteration
\(\P\). Therefore we have at least \(\kappa\) many reals in the extension \(V[G]\). 
Since the forcing is ccc of size \(\kappa\) and \(\kappa\) has uncountable 
cofinality, a nice-name argument shows that the continuum equals \(\kappa\) in the 
extension.

To see that \ma fails in \(V[G]\), notice that \(\P\) is exactly the type of iteration
considered in lemma~\ref{lemma:FSIterationFactorsAsProduct}. The lemma then implies
that \(\P\) is equivalent to \(\overline{\P}\times\Add(\omega,1)\) for some 
\(\overline{\P}\). Therefore the extension
\(V[G]\) is obtained by adding a Cohen real to some intermediate model.
But, as CH fails in the final extension, Roitman has shown 
in~\cite{Roitman1979:AddingRandomOrCohenRealEffectMA} that Martin's Axiom
must also fail there.

Finally, we show that the grounded Martin's Axiom holds in \(V[G]\) with
\(V\) as the ground model. Before we consider the general case let us look at
the easier situation when \(\kappa\) is regular. 
Thus, let \(\Q\in V\) be a poset of size less than
\(\kappa\) which is ccc in \(V[G]\) and \(\mathcal{D}\in V[G]\) a family of fewer
than \(\kappa\) many dense subsets of \(\Q\). We can assume without loss of 
generality that \(\Q\in H_\kappa\). Code the elements of \(\mathcal{D}\) into a
single set \(D\subseteq \kappa\) of size \(\lambda<\kappa\). Since \(\P\) is ccc, 
the set \(D\) has a nice \(\P\)-name \(\dot{D}\) of size \(\lambda\). Since the iteration 
\(\P\) has finite support, the set \(D\) appears before the end of the iteration.
As we have argued 
before, up to isomorphism, the poset \(\Q\) appears \(\kappa\) many times in the
well-order \(\triangleleft\).
Since \(\Q\) is ccc in \(V[G]\) and \(\P\) is ccc, posets isomorphic to \(\Q\) will
be forced with unboundedly often in the iteration \(\P\) and, therefore, eventually
a \(\mathcal{D}\)-generic will be added for \(\Q\).

If we now allow \(\kappa\) to be singular we run into the problem that the dense sets
in \(\mathcal{D}\) might not appear at any initial stage of the iteration \(\P\).
We solve this issue by using lemma~\ref{lemma:FSIterationFactorsAsProduct} to
factor a suitable copy of \(\Q\) out of the iteration \(\P\) and see it as coming
after the forcing that added \(\mathcal{D}\).

Let \(\Q,\mathcal{D},D,\lambda\) and \(\dot{D}\) be as before. As mentioned, it may
no longer be true that \(\mathcal{D}\) appears at some initial stage of the iteration.
Instead, note that, since \(\dot{D}\) has size \(\lambda\), there are at most
\(\lambda\) many indices \(\alpha\) such that some condition 
appearing in \(\dot{D}\) has a nontrivial \(\alpha\)-th coordinate.
It follows that there is a \(\delta<\kappa\) such that no 
condition appearing in \(\dot{D}\) has a nontrivial \(\delta\)-th coordinate
and the poset considered at stage \(\delta\) is isomorphic to \(\Q\).
Additionally, if we fix a condition \(p\in G\) forcing that
\(\Q\) is ccc in \(V[G]\), we can find such a \(\delta\) beyond the support of \(p\).
Now argue for a moment in \(V[G_\delta]\). 
In this intermediate extension the quotient iteration
\(\P^\delta=\P\rest [\delta,\kappa)\) is of the type considered in 
lemma~\ref{lemma:FSIterationFactorsAsProduct} and, since we chose \(\delta\) beyond
the support of \(p\), we also get that \(\P^\delta\) forces that \(\check{\Q}\) is ccc.
The lemma now implies that \(\P^\delta\) is equivalent to \(\overline{\P}\times\Q\)
for some \(\overline{\P}\). Moving back to \(V\), we can conclude that
\(\P\rest p\) factors as \(\P_\delta\rest p *(\overline{\P}\times\check{\Q})\equiv
(\P_\delta\rest p*\overline{\P})\times\Q\) and obtain the corresponding
generic \((G_\delta*\overline{G})\times H\). Furthermore, the name \(\dot{D}\) is
essentially a \(\P_\delta*\overline{\P}\)-name, since no condition in \(\dot{D}\) has
a nontrivial \(\delta\)-th coordinate. It follows that the set \(D\) appears
already in \(V[G_\delta*\overline{G}]\) and that the final generic
\(H\in V[G_\delta*\overline{G}][H]=V[G]\) is \(\mathcal{D}\)-generic for \(\Q\).
We have thus shown that \((*)\) from lemma~\ref{lemma:grMASmallPosets} holds in
\(V[G]\), which implies that the grounded Martin's Axiom also holds.
\end{proof}

We should reflect briefly on the preceding proof.
If we would have been satisfied with obtaining a model with a regular continuum,
the usual techniques would apply. Specifically, if \(\kappa\) were regular then all the 
dense sets in \(\mathcal{D}\) would have appeared by
some stage of the iteration, after which we would have forced with (a poset isomorphic 
to) \(\Q\), yielding the desired \(\mathcal{D}\)-generic. This approach, however, fails
if \(\kappa\) is singular, as small sets might not appear before the end of the 
iteration. 
Lemma~\ref{lemma:FSIterationFactorsAsProduct} was key in resolving this issue, allowing 
us to factor the iteration \(\P\) as a product and seeing the forcing \(\Q\) as happening
at the last stage, after the dense sets had already appeared.
The lemma implies that the iteration factors at any stage where
we considered an absolutely ccc poset (for example, we can factor out any Knaster
poset from the ground model). However, somewhat fortuitously, we can also factor out
any poset to which we might apply the grounded Martin's Axiom, at least below
some condition. There is no surrogate for lemma~\ref{lemma:FSIterationFactorsAsProduct}
for the usual Solovay-Tennenbaum iteration and indeed, Martin's Axiom implies that
the continuum is regular.

\begin{corollary}
\label{cor:grMAConsistentWithSuslinTrees}
The grounded Martin's Axiom is consistent with the existence of a Suslin tree.
\end{corollary}

\begin{proof}
We saw that the model of \grma constructed in the above proof was obtained by
adding a Cohen real to an intermediate extension. 
Adding that Cohen real also adds a Suslin tree by a result of
Shelah~\cite{Shelah1984:TakeSolovayInaccessibleAway}.
\end{proof}

A further observation we can make is that the cofinality of \(\kappa\) plays no
role in the proof of theorem~\ref{thm:CanonicalgrMA} beyond the obvious König's 
inequality requirement on the
value of the continuum. This allows us to obtain models of the grounded Martin's
Axiom with a singular continuum and violate cardinal arithmetic properties
which must hold in the presence of Martin's Axiom.

\begin{corollary}
The grounded Martin's Axiom is consistent with \(2^{<\c}>\c\).
\end{corollary}

\begin{proof}
Starting from some model, perform the construction of theorem~\ref{thm:CanonicalgrMA}
with \(\kappa\) singular. In the extension the continuum equals \(\kappa\).
But the desired inequality is true in any model where the continuum is singular:
of course \(\c=2^\omega\leq 2^{\cf(\c)}\leq 2^{<\c}\) is true but equalities cannot hold 
since the middle two cardinals have different cofinalities by König's inequality.
\end{proof}

On the other hand, assuming we start with a model satisfying GCH, the model of 
theorem~\ref{thm:CanonicalgrMA} will satisfy the best possible alternative
to \(2^{<\c}=\c\), namely \(2^{<\cf(\c)}=\c\). Whether this always happens remains open.

\begin{question}
\label{q:grMAWeakLuzin}
Does the grounded Martin's Axiom imply that \(2^{<\cf(\c)}=\c\)?
\end{question}
\note{Do we have \(2^{<\cf(\c)}=\c\)? Yes in canonical model, but not clear in
general; does not follow from \mac (blow up \(\c\) but stay below \(2^{\omega_1}\)).}

%Another point of interest is that the model built in 
%theorem~\ref{thm:CanonicalgrMA} exhibits a high degree of indestructibility
%for the grounded Martin's Axiom (compare this with 
%theorems~\ref{thm:CohenPreservesgrMA} and \ref{thm:RandomGivesgrMA}).
%
%\begin{corollary}
%Let \(\mathcal{A}\) be a family of absolutely ccc posets (i.e.\ posets which remain
%ccc in any ccc forcing extension). Then there is a ccc forcing extension satisfying
%\grma in which this axiom is indestructible by forcing with any poset from 
%\(\mathcal{A}\).
%\end{corollary}
%
%\begin{proof}
%The forcing extension is just the extension by the poset \(\P\) from 
%theorem~\ref{thm:CanonicalgrMA}. We have argued there that any element of
%\(\mathcal{A}\) will appear unboundedly often in the enumeration \(\Q_\alpha\)
%and, since these copies are all absolutely ccc, they will be included as factors of 
%the product \(\P\). It follows that if \(\Q\) is any poset in \(\mathcal{A}\)
%then \(\P\times\Q\) is isomorphic to \(\P\), which gives the indestructibility
%result.
%\end{proof}

\section{The axiom's relation to other fragments of Martin's Axiom}

Let us now compare some of the combinatorial consequences of the grounded Martin's 
Axiom with those of the usual Martin's Axiom. We first make an easy observation.

\begin{proposition}
\label{prop:grMAImpliesMACountable}
The local grounded Martin's Axiom implies \mac.
\end{proposition}

\begin{proof}
Fix the cardinal \(\kappa\geq\c\) and the \(\mathrm{ZFC}^-\) ground model 
\(M\subseteq H_{\kappa^+}\) witnessing local \grma. 
Observe that the the model \(M\) contains the poset \(\mathrm{Add}(\omega,1)\), since
its elements are effectively coded by the natural numbers.
This poset is therefore always a valid target for local \grma.
\end{proof}

It follows from the above proposition that the (local) grounded Martin's Axiom
will have some nontrivial effects on the cardinal characteristics of the continuum.
In particular, we obtain the following.

\begin{corollary}
\label{cor:grMALargeCharacts}
The local grounded Martin's Axiom implies that the cardinals on the right side of 
Cichoń's diagram equal the continuum.
In particular, this holds for both the covering number for category 
\(\mathbf{cov}(\mathcal{B})\) and the reaping number \(\mathfrak{r}\).
\end{corollary}

\begin{proof}
All of the given equalities follow already from \mac; we briefly summarize
the arguments from~\cite{Blass2010:CardinalCharacteristicsHandbook}.

The complement of any nowhere dense subsets of the real line is dense. It follows that, given
fewer than continuum many nowhere dense sets, we can apply \mac to obtain a real number
not contained in any of them. Therefore the real line cannot be covered by fewer
than continuum many nowhere dense sets and, consequently, also not by fewer than
continuum many meagre sets.

To see that the reaping number must be large, observe that, given any infinite
\(x\subseteq\omega\), there are densely many conditions in \(\mathrm{Add}(\omega,1)\)
having arbitrarily large intersection with both \(x\) and \(\omega\setminus x\).
It follows that a Cohen real will split \(x\). Starting with fewer than continuum many
reals and applying \mac, we can therefore find a real splitting all of them, which
means that the original family was not a reaping family.
\end{proof}

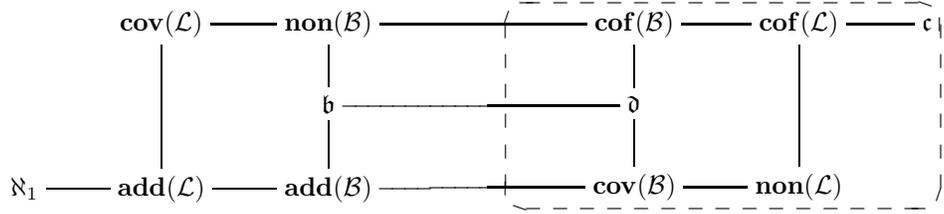
\begin{figure}[ht]
\[\scalebox{1}{

\xy
\xymatrix"*"@R-0.6pc{
\ar@{-}[]+<-1em,0em>;[r]&
\mathbf{cof}(\mathcal{B})\ar@{-}[r]\ar@{-}[d] & 
\mathbf{cof}(\mathcal{L})\ar@{-}[r]\ar@{-}[dd] & \c \\
\ar@{-}[]+<-1em,0em>;[r]&
\mathfrak{d}\ar@{-}[d] &&\\
\ar@{-}[]+<-1em,0em>;[r]&
\mathbf{cov}(\mathcal{B})\ar@{-}[r] & 
\mathbf{non}(\mathcal{L})
}
\drop\frm<8pt>{--}
\POS-(65,0)
\xymatrix{
& \mathbf{cov}(\mathcal{L})\ar@{-}[r]\ar@{-}[dd] & 
\mathbf{non}(\mathcal{B})\ar@{-}["*"ll]\ar@{-}[d] \\
&& \mathfrak{b}\ar@{-}[d]\ar@{-}["*"ll] \\
\aleph_1\ar@{-}[r] & \mathbf{add}(\mathcal{L})\ar@{-}[r] & 
\mathbf{add}(\mathcal{B})\ar@{-}["*"ll]
}
\endxy
}\]
\end{figure}

But where Martin's Axiom strictly prescribes the size of all cardinal characteristics 
of the continuum, the grounded Martin's Axiom allows for more leeway in some cases. 
Observe that, since \(\kappa>\omega_1\), the iteration
\(\P\) of theorem~\ref{thm:CanonicalgrMA} contains \(\mathrm{Add}(\omega,\omega_1)\)
as an iterand. 
Thus, by lemma~\ref{lemma:FSIterationFactorsAsProduct}, there is a 
poset \(\overline{\P}\) such that \(\P\) is equivalent 
to \(\overline{\P}\times\mathrm{Add}(\omega,\omega_1)\).

\begin{theorem}
\label{thm:grMASmallCharacts}
It is consistent that the grounded Martin's Axiom holds, \textup{CH} fails and the cardinal
characteristics on the left side of Cichoń's diagram, as well as the splitting
number \(\mathfrak{s}\) are equal to \(\aleph_1\).
\end{theorem}

\note{Can we separate these characteristics under \grma?}

\begin{proof}
Consider a model \(V[G]\) of \grma satisfying \(\c>\aleph_1\) which was obtained by 
forcing  with the iteration \(\P\) from theorem~\ref{thm:CanonicalgrMA} over a model of 
GCH. We  have argued that this model is obtained by adding \(\aleph_1\) many Cohen reals 
to some intermediate extension. We again briefly summarize the standard arguments
for the smallness of the indicated cardinal characteristics in such an extension
(see~\cite{Blass2010:CardinalCharacteristicsHandbook} for details).

Let \(X\) be the set of \(\omega_1\) many Cohen reals added by the final stage of 
forcing. We claim it is both nonmeager and splitting. Note that any real in
\(V[G]\) appears before all of the Cohen reals in \(X\) have appeared. It follows
that every real in \(V[G]\) is split by some real in \(X\). Furthermore, if \(X\) were
meager, it would be contained in a meager Borel set, whose Borel code also
appears before all of the reals in \(X\) do. But this leads to contradiction, since
any Cohen real will avoid any meager set coded in the ground model.
\end{proof}

To summarize, while the grounded Martin's Axiom implies that the right side
of Cichoń's diagram is pushed up to \(\c\), it is consistent with the left side dropping
to \(\aleph_1\) (while CH fails, of course). This is the most extreme
way in which the effect of the grounded Martin's Axiom on Cichoń's diagram can differ 
from that of Martin's Axiom. The
precise relationships under \grma between the cardinal characteristics on the left
warrant further exploration in the future. 

We can consider further the position of the grounded Martin's Axiom within the 
hierarchy of the more well-known fragments of Martin's Axiom. As we have already 
mentioned, (local) \grma implies \mac. We can strengthen this slightly. Let \macoh
denote Martin's Axiom restricted to posets of the form \(\Add(\omega,\lambda)\) for
some \(\lambda\). It will turn out that local \grma also implies \macoh.

\begin{lemma}
\label{lemma:MACohenSmallPosets}
The axiom \macoh is equivalent to its restriction to posets of the form
\(\Add(\omega,\lambda)\) for \(\lambda<\c\).
\end{lemma}

\begin{proof}[Proof\;\footnotemark]
\footnotetext{The proof of the key claim was suggested by Noah Schweber.}
Let \(\P=\Add(\omega,\kappa)\) and fix a collection \(\mathcal{D}\) of \(\lambda<\c\)
many dense subsets of \(\P\). As usual, let \(\Q\) be an elementary substructure of
\(\langle\P,D\rangle_{D\in\mathcal{D}}\) of size \(\lambda\). We shall show that
\(\Q\) is isomorphic to \(\Add(\omega,\lambda)\). The lemma then follows easily.

To demonstrate the desired isomorphism we shall show that \(\Q\) is determined
by the single-bit conditions it contains. More precisely, \(\Q\) contains precisely those
conditions which are meets of finitely many single-bit conditions in \(\Q\).

First note that being a single-bit conditions is definable in \(\P\): 
these are precisely
the coatoms of the order. Furthermore, given a coatom \(p\), its complementary 
coatom \(\bar{p}\) with the single bit flipped is definable from \(p\) as the unique
coatom such that any condition is compatible with either \(p\) or \(\bar{p}\). It follows
by elementarity that the coatoms of \(\Q\) are precisely the single-bit conditions 
contained in \(\Q\) and that \(\Q\) is closed under the operation \(p\mapsto\bar{p}\).
In \(\P\) any finite collection of pairwise compatible coatoms has a meet, therefore
the same holds in \(\Q\) and the meets agree. Conversely, any given condition in \(\P\)
uniquely determines the finitely many coatoms it strengthens and therefore all the
coatoms determined by conditions in \(\Q\) are also in \(\Q\). Taken together, 
this proves the claim.

It follows immediately from the claim that \(\Q\) is isomorphic to \(\Add(\omega,|X|)\)
where \(X\) is the set of coatoms of \(\Q\), and also that \(|X|=|\Q|\).
\end{proof}

\begin{proposition}
\label{prop:grMAImpliesMACohen}
The local grounded Martin's Axiom implies \macoh.
\end{proposition}

\begin{proof}
Suppose the local grounded Martin's Axiom holds, witnessed by \(\kappa\geq\c\) and a
\(\mathrm{ZFC}^-\)-model \(M\subseteq H_{\kappa^+}\). 
In particular, the height of \(M\) is \(\kappa^+\) and
\(M\) contains all of the posets \(\Add(\omega,\lambda)\) for \(\lambda<\kappa^+\).
But this means that Martin's Axiom holds for all the posets \(\Add(\omega,\lambda)\)
where \(\lambda<\c\) and lemma~\ref{lemma:MACohenSmallPosets} now implies that
\macoh holds.
\end{proof}

\begin{figure}[h!t]
\[
\xymatrix{
&\ma\ar@{-}[dl]\ar@{-}[dr]&\\
\mak\ar@{-}[d]&&\grma\ar@{-}[d]\\
\textup{MA(\(\sigma\)-linked)}\ar@{-}[d]&&\text{local \grma}\ar@{-}[ddl]\\
\mas\ar@{-}[dr]\\
&\mathrm{MA(Cohen)}\ar@{-}[d]\\
&\mathrm{MA(countable)}}
\]
\end{figure}

As we have seen, the local grounded Martin's Axiom implies some of the weakest
fragments of Martin's Axiom. Theorem~\ref{thm:grMASmallCharacts} tells us, 
however, that this behaviour stops quite quickly.

\begin{corollary}
\label{cor:grMANotImpliesMAsigma}
The grounded Martin's Axiom does not imply \mas.
\end{corollary}
\begin{proof}
By theorem~\ref{thm:grMASmallCharacts} there is a model of the grounded Martin's Axiom
where the bounding number is strictly smaller than the continuum. But this is impossible
under \mas, since applying the axiom to Hechler forcing yields for any family of
fewer than continuum many reals a real dominating them all.
\end{proof}

As mentioned earlier, Reitz has shown that we can perform class forcing over
any model in such a way that the resulting extension has the same \(H_\c\)
and is also not a set-forcing extension of any ground model. Performing this
construction over a model of \mas (or really any of the standard fragments of
Martin's Axiom) shows that \mas does not imply \grma, for the
disappointing reason that the
final model is not a ccc forcing extension of anything.
However, it turns out that already local \grma is independent of \mas, and even
of \mak. 
This places the grounded Martin's Axiom, as well as its local version, outside the usual 
hierarchy of fragments of Martin's Axiom.

\begin{theorem}
\label{thm:grMASeparateFragment}
Assume \(V=L\) and let \(\kappa>\omega_1\) be a regular cardinal.
Then there is a ccc forcing extension which satisfies
\(\mak+\mathfrak{c}=\kappa\) and in which the local grounded Martin's Axiom fails.
\end{theorem}

\begin{proof}
Let \(\P\) be the usual finite-support iteration forcing 
\(\mak+\mathfrak{c}=\kappa\). More precisely, we consider
the names for posets in \(H_\kappa\) using appropriate bookkeeping and append them
to the iteration if they, at that stage, name a Knaster poset.
Let \(G\subseteq \P\) be generic. We claim that the local grounded Martin's Axiom 
fails in the extension \(L[G]\).

Notice first that \(\P\), being a finite-support iteration of Knaster
posets, is Knaster. It follows that the product of \(\P\) with any ccc poset is
still ccc. In particular, forcing with \(\P\) preserves the Suslin trees of \(L\).

Now fix a \(\lambda\geq\kappa\) and let \(M\in L[G]\) be a transitive \(\mathrm{ZFC}^-\)
model of height \(\lambda^+\). It is straightforward to see that \(M\) builds
its constructible hierarchy correctly so that, in particular, \(L_{\omega_2}\subseteq
M\). This implies that \(M\) has all of the Suslin trees of \(L\).
Since these trees are still Suslin in \(L[G]\), partially generic filters do not
exist for them and the model \(M\) does not witness
local \grma in \(L[G]\). As \(\lambda\) and \(M\) were completely arbitrary,
local \grma must fail in \(L[G]\).
%Now suppose that the local grounded Martin's Axiom held in \(L[G]\).
%This means that there is some \(\lambda\geq\kappa\) such that 
%\(H_{\lambda^+}^{L[G]}\) is a ccc forcing extension of some transitive 
%\(\mathrm{ZFC}^-\) model \(M\). Let us assume for the moment that \(M\) contains
%a Suslin tree of \(L\). Then we can apply \grma to one of these trees in \(L[G]\)
%to conclude that it is no longer Suslin, giving a contradiction with the previous
%paragraph.
%
%To conclude our proof we need to show that \(M\) actually has some of \(L\)'s trees.
%In fact, \(M\) builds its \(L\) hierarchy correctly up to its own ordinal 
%height  \(\kappa^+\), so that, in particular, \(L_{\omega_2}\subseteq M\).
%This is tedious, but straightforward to show. 
%Briefly, given \(L_\alpha\in M\), we can apply
%the axiom of separation (using a truth predicate for \(L_\alpha\) and parameters
%from \(L_\alpha^{<\omega}\), which can both be constructed by \(M\)) to obtain all
%of the definable subsets of \(L_\alpha\) in a uniform way. The axiom of collection
%then allows us to collect these into \(L_{\alpha+1}\in M\).
%At limit stages we similarly simply collect all of the previous stages together.
\end{proof}

Let us mention that it is quite
easy to perform ccc forcing over any model and have \grma fail in the extension.

\begin{corollary}
\label{cor:grMAFailsAfterHechler}
Given any model \(V\) there is a ccc forcing extension \(V[G]\) in which
the local grounded Martin's Axiom fails.
\end{corollary}

\begin{proof}
We may assume that CH fails in \(V\). If \(\P\) is the length \(\omega_1\) finite-support
iteration of Hechler forcing and \(G\subseteq \P\) is generic then it is easily seen
that \(G\) is a dominating family in \(V[G]\) and therefore 
the dominating number of \(V[G]\) equals \(\aleph_1\).
It now follows from 
corollary~\ref{cor:grMALargeCharacts} that the local \grma fails in \(V[G]\).
\end{proof}

In the following two sections we shall explore the other side of the coin:
\grma is preserved by certain kinds of ccc forcing.

\section{Adding a Cohen real to a model of the grounded Martin's Axiom}

An interesting question when studying fragments of Martin's Axiom is what effect 
adding various kinds of generic reals has on it. It was shown by 
Roitman~\cite{Roitman1979:AddingRandomOrCohenRealEffectMA} that
\(\ma_{\aleph_1}\) is destroyed after adding a Cohen or a random real. At the same 
time, it was shown that adding a Cohen real preserves a certain fragment, \mas. 
In this section we follow the spirit of Roitman's arguments to show that
the grounded Martin's Axiom is preserved, even with respect to the same ground model, 
after adding a Cohen real.

It is well known that \(\ma+\lnot\mathrm{CH}\) implies that any ccc poset is Knaster 
(recall that a poset \(\P\) is Knaster if any uncountable subset of \(\P\) has in turn an 
uncountable subset of pairwise compatible elements). We start this section by transposing 
this fact to the \grma setting.

\begin{lemma}
\label{lemma:grMACccKnaster}
Let \(V\) satisfy the local grounded Martin's Axiom over the ground model \(M\subseteq
H_{\kappa^+}\) and suppose CH fails in \(V\).
Then any poset \(\P\in M\) which is ccc in \(V\) is Knaster in \(V\).
\end{lemma}

\begin{proof}
Let \(\P\) be as in the statement of the lemma and let 
\(A=\set{p_\alpha}{\alpha<\omega_1}\in V\)
be an uncountable subset of \(\P\). We first claim that there is a \(p^*\in \P\) such that
any \(q\leq p^*\) is compatible with uncountably many elements of \(A\). For suppose not.
Then there would be for any \(\alpha<\omega_1\) some \(q_\alpha\leq p_\alpha\) which was
compatible with only countably many elements of \(A\). We could thus choose 
\(\beta(\alpha)<\omega_1\) in such a way that \(q_\alpha\) would be incompatible with any
\(p_\beta\) for \(\beta(\alpha)\leq\beta\). Setting \(\beta_\alpha=\beta^\alpha(0)\)
(meaning the \(\alpha\)-th iterate of \(\beta\)), 
this would mean that \(\set{q_{\beta_\alpha}}{\alpha<\omega_1}\) is an uncountable antichain
in \(\P\), contradicting the fact that \(\P\) was ccc in \(V\).

By replacing \(\P\) with the cone below \(p^*\) and modifying \(A\) appropriately, we may 
assume that in fact every element of \(\P\) is compatible with uncountably many elements of
\(A\). We now let
\(D_\alpha=\bigcup_{\beta\leq\alpha}\P\rest p_\beta\)
for \(\alpha<\omega_1\). The sets \(D_\alpha\) are dense in \(\P\) and by 
\(\grma+\lnot\mathrm{CH}\), we can find, in \(V\), a filter \(H\subseteq \P\) which 
intersects every \(D_\alpha\). But then \(H\cap A\) is an uncountable set of pairwise 
compatible elements.
\end{proof}

We now introduce the main technical device we will use in showing that the grounded
Martin's Axiom is preserved when adding a Cohen real. In the proof we will be 
dealing with a two step extension \(W\subseteq W[G]\subseteq
W[G][c]\) where the first step is some ccc extension, the second adds a Cohen
real and \(W[G]\) satisfies the grounded Martin's Axiom over \(W\).
To utilize the forcing axiom in \(W[G]\) in verifying it in \(W[G][c]\), we need
to find a way of dealing with (names for) dense sets from \(W[G][c]\) in \(W[G]\).
The termspace forcing construction (due to Laver and possibly independently also 
Abraham, Baumgartner, Woodin and others) comes to mind (for more information
on this construction we point the reader to~\cite{Foreman1983:SaturatedIdeals}), 
however the
posets arising from this construction are usually quite far from being ccc and
are thus unsuitable for our context. We attempt to rectify the situation by
radically thinning out the full termspace poset and keeping only the simplest
conditions.

\begin{definition}
Given a poset \(\P\), a \(\P\)-name \(\tau\) will be called a 
\emph{finite \(\P\)-mixture} if there exists a finite
maximal antichain \(A\subseteq \P\) such that for every \(p\in A\) there is some 
\(x\) satisfying \(p\forces_\P \tau=\check{x}\). The antichain \(A\) is called a 
\emph{resolving antichain} for \(\tau\) and we denote the value \(x\) of \(\tau\)
at \(p\) by \(\tau^p\).
\end{definition}

%Suppose \(\P\) is a poset. 
%We shall call a \(\P\)-name \(\dot{\Q}\) \emph{well-based}
%if \(\dot{\Q}\) names a poset whose domain is a fixed set in the ground model (for
%example, \(\dot{\Q}\) might name a relation on some fixed ordinal). 

\begin{definition}
Let \(\P\) and \(\Q\) be posets. The
\emph{finite-mixture term\-space poset} for \(\Q\) over \(\P\) is
\[\Termfin(\P,\Q)=\set{\tau}{\textup{\(\tau\) is a finite \(\P\)-mixture and
\(1\forces_\P\tau\in\check{\Q}\)}}\]
ordered by letting \(\tau\leq\sigma\) iff \(1\forces_\P\tau\leq_{\Q}\sigma\).
\end{definition}

%The restriction to well-based \(\dot{\Q}\) is not particularly significant. 
%It serves only to make it possible for
%finite \(\P\)-mixtures to name an element of \(\dot{\Q}\) and any poset in \(V^\P\)
%has an isomorphic copy which has a well-based name.

As a side remark, let us point out that in all interesting cases the finite-mixture 
termspace poset is not a regular suborder of the full termspace poset and we can expect 
genuinely different properties. In fact, this occurs as soon as \(\P\) and \(\Q\) are 
nontrivial. To see this, suppose \(\set{p_n}{n<\omega}\) and \(\set{q_n}{n<\omega}\) are 
infinite antichain in \(\P\) and \(\Q\), respectively. By mixing we can find a
\(\tau\in\Term(\P,\Q)\) such that \(p_n\forces\tau=q_n\); we claim that this \(\tau\)
does not have a reduction to \(\Termfin(\P,\Q)\). 
Suppose \(\sigma\in\Termfin(\P,\Q)\) were such a reduction. Then there is 
a condition \(p\) in its resolving antichain that
is compatible with at least two conditions \(p_i\), say \(p_0\) and \(p_1\).
Since \(\sigma\) is a reduction of \(\tau\), it and all stronger conditions are
compatible with \(\tau\), and this means that \(\sigma^p\) is compatible with
\(q_0\) and \(q_1\). Let \(q'\leq\sigma^p,q_1\) and define
a strengthening \(\sigma'\leq\sigma\) by setting \(\sigma'^p=q'\) and keeping the rest of
the mixture the same as in \(\sigma\). But now \(\sigma'\) and \(\tau\) are clearly
incompatible.

In what follows, let us write \(\C=\funcs{<\omega}{2}\).
We should mention two key issues with the finite-mixture termspace poset 
construction. 
Firstly, the construction is very sensitive to
the concrete posets being used. For example, the forthcoming 
lemma~\ref{lemma:CohenFiniteMixturesKnaster} will show that \(\Termfin(\C,\C)\) is
Knaster, but it is not difficult to see that \(\Termfin(\mathrm{ro}(\C),\C)\) already
has antichains of size continuum. Therefore we cannot freely substitute forcing
equivalent posets in the construction. In fact, if \(\B\) is a complete Boolean
algebra then one easily sees that \(\Termfin(\B,\Q)\) consists of exactly those names
that have only finitely many interpretations and, if \(\Q\) is nontrivial
and \(\B\) has no atoms, this poset will have antichains of size continuum.
The second issue is that
it is quite rare for a poset to have a large variety of finite maximal antichains.
Some, such as \(\funcs{<\omega}{\omega}\) or various collapsing posets, have none
at all except the trivial one-element maximal antichain, while others, such as 
\(\funcs{<\omega_1}{2}\), have a few, but they do not capture the structure of
the poset very well. In all of these cases we do not expect the finite-mixture
termspace poset to be of much help. Nevertheless, in the case of Cohen forcing
\(\C\) it turns out to be a useful tool.

\begin{lemma}
\label{lemma:CohenFiniteMixturesKnaster}
If \(\Q\) is a Knaster poset then \(\Termfin(\C,\Q)\) is also Knaster.
\end{lemma}

\begin{proof}
Let \(T=\set{\tau_\alpha}{\alpha<\omega_1}\) be an uncountable subset of 
\(\Termfin(\C,\Q)\) and
choose resolving antichains \(A_\alpha\) for \(\tau_\alpha\). By refining the 
\(A_\alpha\)
we may assume that each of them is a level of the tree \(\C\) and, by thinning out \(T\)
if necessary, that they are all in fact the same level \(A\). Let us enumerate 
\(A=\{s_0,\dotsc,s_k\}\) and write \(\tau_\alpha^i\) instead of \(\tau_\alpha^{s_i}\). 

Since \(\Q\) is Knaster, an uncountable subset \(Z_0\) of \(\omega_1\) such that
the set \(\set{\tau_\alpha^0}{\alpha\in Z_0}\subseteq \Q\) 
consists of pairwise compatible elements. Proceeding
recursively, we can find an uncountable \(Z\subseteq\omega_1\) such that for every
\(i\leq k\) the set \(\set{\tau_\alpha^i}{\alpha\in Z}\subseteq \Q\) 
consists of pairwise compatible elements. 
We can mix the lower bounds of \(\tau_\alpha^i\) and \(\tau_\beta^i\) over the antichain
\(A\) to produce a name \(\sigma_{\alpha\beta}\in \Termfin(\C,\Q)\) such that
\(\sigma_{\alpha\beta}\) is a lower bound for \(\tau_\alpha\) and \(\tau_\beta\).
Thus \(\set{\tau_\alpha}{\alpha\in Z}\) is an uncountable subset of \(T\) 
consisting of pairwise compatible elements, which proves that \(\Termfin(\C,\Q)\) is 
Knaster. 
\end{proof}

The following lemma is somewhat awkward, but it serves to give us a way of
transforming a name for a dense subset of \(\Q\) into a closely related 
actual dense subset of \(\Termfin(\C,\Q)\). With the usual termspace forcing
construction simply taking \(E=\set{\tau}{\forces\tau\in\dot{D}}\) would have sufficed,
but this set is not dense in \(\Termfin(\C,\Q)\), so modifications are necessary.

\begin{lemma}
\label{lemma:CohenFiniteMixturesDenseTranslation}
Let \(\Q\) be poset and \(\dot{D}\) a \(\C\)-name for a
dense subset of \(\Q\). Then for any \(n<\omega\) the set
\begin{align*}
E_n=\{\,\tau\in \Termfin(\C,\Q)\,;\,& \exists A \textup{ a resolving antichain for \(\tau\)}\,
\forall s\in A\colon \\&
n\leq |s| \land \exists s'\leq s\colon s'\forces\tau\in\dot{D}\,\}
\end{align*}
is a dense subset of \(\Termfin(\C,\Q)\).
\end{lemma}

One can think of the set \(E_n\) as the set of those \(\tau\) that have a sufficiently
deep resolving antichain, none of whose elements force \(\tau\) to not be in
\(\dot{D}\).

\begin{proof}
Let \(\sigma\in \Termfin(\C,\Q)\) and let \(A\) be a resolving antichain for it. Any finite refinement
of a resolving antichain is, of course, another resolving antichain, so we may assume that
we already have \(n\leq |s|\) for all \(s\in A\). By fullness
we can find a name \(\rho\) for an element of \(\Q\) such that
\(1\forces_{\C}(\rho\leq \sigma\land \rho\in\dot{D})\). For each \(s\in A\) we can find
an \(s'\leq s\) such that \(s'\forces_{\C} \rho=\check{q}_s\)
for some \(q_s\in\Q\). By mixing the \(q_s\) over the antichain 
\(A\), we get a name \(\tau\in E_n\) such that \(\tau\leq\sigma\), 
which shows that \(E\) is dense in \(\Termfin(\C,\Q)\).
\end{proof}

%\begin{lemma}
%\label{lemma:CohenFiniteMixtures}
%Let \(\dot{\Q}\) be a well-based \(\C\)-name.
%\begin{enumerate}
%\item
%If \(\dot{\Q}\) is forced to be Knaster then \(\Termfin(\C,\dot{\Q})\) is Knaster.
%\item
%%Assume \(\Q\) is separative and atomless. 
%If \(\dot{D}\) is a \(\C\)-name for a dense subset of \(\dot{\Q}\) then for any 
%\(n<\omega\)
%\begin{align*}
%E_n=\{\,\tau\in \Termfin(\C,\dot{\Q})\,;\,& \exists A \textup{ a resolving antichain for \(\tau\)}\,
%\forall s\in A\colon \\&
%n\leq |s| \land \exists s'\leq s\colon s'\forces\tau\in\dot{D}\,\}
%\end{align*}
%is a dense subset of \(\Termfin(\C,\dot{\Q})\).
%\end{enumerate}
%\end{lemma}

\begin{theorem}
\label{thm:CohenPreservesgrMA}
Assume the local grounded Martin's Axiom holds in \(V\) over the ground model 
\(M\subseteq H_{\kappa^+}\) and let 
\(V[c]\) be obtained by adding a Cohen real to \(V\). Then \(V[c]\) also satisfies 
the local grounded Martin's Axiom over the ground model \(M\subseteq H_{\kappa^+}^{V[c]}=
H_{\kappa^+}[c]\).
\end{theorem}

\begin{proof}
By assumption there is a ccc poset \(\P\in M\) such that \(H_{\kappa^+}^V=M[G]\) for 
an \(M\)-generic \(G\subseteq \P\).
We may assume that \(\mathrm{CH}\) fails in \(V\), for otherwise it would also hold in
the final extension\(V[c]\), which would then
satisfy the full Martin's Axiom. Consider a poset \(\Q\in M\) which is ccc in \(V[c]\). 
Since \(\C\) is ccc, \(\Q\) must also
be ccc in \(V\) and by lemma~\ref{lemma:grMACccKnaster} is in fact Knaster in \(V\).

Let \(\lambda<\c^{V[c]}\) be a cardinal and let 
\(\mathcal{D}=\set{D_\alpha}{\alpha<\lambda}\in V[c]\) be a collection of dense 
subsets of \(\Q\). Pick names \(\dot{D}_\alpha\in V\) for these such that 
\(1\forces_{\C} \text{``\(\dot{D}_\alpha\subseteq \check{\Q}\) is dense''}\).

Consider \(\R=\Termfin(\C,\Q)\in M\). Note that \(\R\) is computed the same in
\(M\) and in \(V\). It now follows from lemma~\ref{lemma:CohenFiniteMixturesKnaster} 
that \(\R\) is Knaster in \(V\) (although not necessarily in \(M\)).

%We claim that it is Knaster in \(W[G]\). So let
%\(S\in W[G]\) be an uncountable subset of \(\R\). Then \(S\) is also a subset of
%\(\R'=\Termfin(\C,\Q)^{W[G]}\). It follows from (1) of lemma~\ref{lemma:CohenFiniteMixtures} that \(\R'\) is
%Knaster in \(W[G]\), so there is an uncountable \(S'\subseteq S\) in \(W[G]\), consisting
%of conditions which are pairwise compatible in \(W[G]\). The key observation now is that, even
%though the poset \(\R'\) is strictly larger than \(\R\), every name in \(\R'\) is forced by
%\(\C\) to be equal to some name in \(\R\), since we are only dealing with finite mixtures.
%This implies that the conditions in \(S'\) are pairwise compatible in \(\R\) and thus
%\(\R\) is Knaster in \(W[G]\).

Let \(E_{\alpha,n}\subseteq \R\) be the dense sets associated to the
\(\dot{D}_\alpha\) as in lemma~\ref{lemma:CohenFiniteMixturesDenseTranslation}.
%; the lemma doesn't
%apply directly, but we can use the trick from the previous paragraph to transfer the dense sets
%from \(\R'\in W[G]\) to \(\R\in W\). 
Write \(\mathcal{E}=\set{E_{\alpha,n}}{\alpha<\lambda,n<\omega}\). Applying
the grounded Martin's Axiom in \(V\), 
we get an \(\mathcal{E}\)-generic filter \(H\subseteq \R\).
We will show that the filter generated by the set
\(H^c=\set{\tau^c}{\tau\in H}\) is \(\mathcal{D}\)-generic. Pick
a \(\dot{D}_\alpha\) and consider the set
\begin{align*}
B_\alpha=\{\,s'\in \C\,;\,&\exists \tau\in H\exists A \text{ a resolving antichain for \(\tau\)}
\,\exists s\in A \colon\\& s'\leq s\land s'\forces_{\C}\tau\in \dot{D}_\alpha\,\}
\end{align*}
We will show that \(B_\alpha\) is dense in \(\C\). To that end, pick a \(t\in \C\). Since
\(H\) is \(\mathcal{E}\)-generic, there is some \(\tau\in H\cap E_{\alpha,|t|}\). Let
\(A\) be a resolving antichain for \(\tau\). Since \(A\) is maximal, \(t\) must be compatible
with some \(s\in A\), and, since \(|t|\leq |s|\), we must in fact have \(s\leq t\). But then,
by the definition of \(E_{\alpha,|t|}\), there exists a \(s'\leq s\) such that
\(s'\forces_{\C}\tau\in \dot{D}\). This exactly says that \(s'\in B_\alpha\)
and also \(s'\leq s\leq t\). Thus \(B_\alpha\) really is dense in \(\C\).

By genericity we can find an \(s'\in B_\alpha\cap c\). If \(\tau\in H\) is the corresponding
name, the definition of \(B_\alpha\) now implies that \(\tau^c\in D_\alpha\cap H^c\). Thus
\(H^c\) really does generate a \(\mathcal{D}\)-generic filter.
\end{proof}

The proof is easily adapted to show that, starting from the full grounded Martin's Axiom 
in \(V\) over a ground model \(W\), we obtain the full grounded Martin's Axiom in
\(V[c]\) over the same ground model \(W\).

\section{Adding a random real to a model of the grounded Martin's Axiom}

Our next goal is to prove a preservation theorem for adding a random real. The machinery
of the proof in the Cohen case will be slightly modified to take advantage of
the measure theoretic structure in this context. 

Recall that a measure algebra is a pair \((\B,m)\) where \(\B\) is a complete Boolean
algebra and \(m\colon\B\to[0,1]\) is a countably additive map such that
\(m(b)=1\) iff \(b=1\).

\begin{definition}
Let \((\B,m)\) be a measure algebra and \(0<\varepsilon<1\). A \(\B\)-name \(\tau\) will be
called an \emph{\(\varepsilon\)-deficient finite \(\B\)-mixture} if there is a finite antichain
\(A\subseteq \B\) such that \(m(\sup A)>1-\varepsilon\) and for every \(w\in A\)
there exists some \(x\) such that \(w\forces_\B\tau=\check{x}\).
The antichain \(A\) is called a \emph{resolving antichain} and we denote the value
\(x\) of \(\tau\) at \(w\) by \(\tau^w\).
\end{definition}

\begin{definition}
Let \((\B,m)\) be a measure algebra, \(\Q\) a poset and \(0<\varepsilon<1\).
The \emph{\(\varepsilon\)-deficient
finite mixture termspace poset} for \(\Q\) over \((\B,m)\) is
\[\Termfin^\varepsilon(\B,\Q)=\set{\tau}{\textup{\(\tau\) is an
\(\varepsilon\)-deficient finite \(\B\)-mixture and \(1\forces_\B\tau\in\check{\Q}\)}}\]
ordered by letting \(\tau\leq\sigma\) iff there are resolving antichains \(A_\tau\)
and \(A_\sigma\) such that \(A_\tau\) refines \(A_\sigma\) and 
\(\sup A_\tau\forces_\B\tau\leq \sigma\).
\end{definition}

The following lemma is the analogue of 
lemma~\ref{lemma:CohenFiniteMixturesDenseTranslation} for \(\varepsilon\)-deficient
finite mixtures.

\begin{lemma}
\label{lemma:DeficientFiniteMixDenseTranslation}
Let \((\B,m)\) be a measure algebra, \(\Q\) a poset and \(0<\varepsilon<1\).
If \(\dot{D}\) is a \(\B\)-name for a dense subset of \(\Q\) then 
\[E=\set{\tau\in\Termfin^\varepsilon(\B,\Q)}{\exists A \textup{ a resolving antichain for \(\tau\)}
\colon\sup A\forces_\B\tau\in\dot{D}}\]
is dense in \(\Termfin^\varepsilon(\B,\Q)\).
\end{lemma}

\begin{proof}
Let \(\sigma\in\Termfin^\varepsilon(\B,\Q)\) and pick a resolving antichain 
\(A=\{w_0,\dotsc,w_n\}\) for it. Let \(\delta=m(\sup A)-(1-\varepsilon)\). 
By fullness there are \(\B\)-names \(\rho_i\) for elements of \(\Q\) such that
\(w_i\forces \rho_i\leq \sigma \land \rho\in \dot{D}\). There are
maximal antichains \(A_i\) below \(w_i\) such that each element of
\(A_i\) decides the value of \(\rho_i\). We now choose finite subsets
\(A_i'\subseteq A_i\) such that \(m(\sup A_i)-m(\sup A_i')<\frac{\delta}{n}\).
Write \(A'=\bigcup_i A_i'\). We then have \(m(\sup A')>1-\varepsilon\).
By mixing we can find a \(\B\)-name \(\tau\) for an element of \(\Q\) which is forced 
by each element of \(A'\) to be equal to the appropriate \(\rho_i\). Thus \(A'\)
is a resolving antichain for \(\tau\) and we have
ensured that \(\tau\) is in \(E\) and \(\sigma\leq \tau\).
\end{proof}

In what follows we let \(\Bnull\) be the random Boolean algebra with the induced
Lebesgue measure \(\mu\). The next lemma is the analogue of
lemma~\ref{lemma:CohenFiniteMixturesKnaster}. 
%but note that we now require
%the poset \(\Q\) to be an element of the ground model, effectively restricting from
%well-based names to only check names.

\begin{lemma}
\label{lemma:RandomDeficientMixKnaster}
Let \(\Q\) be a Knaster poset and \(0<\varepsilon<1\). Then 
\(\Termfin^\varepsilon(\Bnull,\Q)\) is Knaster as well.
\end{lemma}

\begin{proof}
Let \(\set{\tau_\alpha}{\alpha<\omega_1}\) be an uncountable subset of 
\(\Termfin^\varepsilon(\Bnull,\Q)\).
Choose resolving antichains \(A_\alpha\) for the \(\tau_\alpha\).
We may assume that there is a fixed \(\delta\) such that
\(1-\varepsilon<\delta<\mu(\sup A_\alpha)\) for all \(\alpha\).
We may also assume that all of the \(A_\alpha\) have the same size \(n\) and enumerate 
them as \(A_\alpha=\{w_\alpha^0,\dotsc,w_\alpha^{n-1}\}\); we shall write
\(\tau_\alpha^i\) instead of \(\tau_\alpha^{w_\alpha^i}\). 
By inner regularity of the measure
we may assume further that the elements of each \(A_\alpha\) are compact.
Using this and the outer regularity of the measure we can find open neighbourhoods
\(w_\alpha^i\subseteq U_\alpha^i\) such that \(U_\alpha^i\) and \(U_\alpha^j\) are
disjoint for all \(\alpha\) and distinct \(i\) and \(j\) and additionally satisfy
\[\mu(U_\alpha^i\setminus w_\alpha^i)<\frac{\delta-(1-\varepsilon)}{n}\]

Fix a countable basis for the topology. Since the \(w^i_\alpha\) are compact,
we may take the \(U^i_\alpha\) to be finite unions of basic opens. Since there
are only countably many such finite unions, we can assume that there are fixed
\(U^i\) such that \(U^i_\alpha=U^i\) for all \(\alpha\). 

We now obtain 
\begin{align*}
\mu(w^i_\alpha\cap w^i_\beta)&=
\mu(w^i_\alpha)-\mu(w_\alpha^i\cap(U^i\setminus w^i_\beta))\\
&\geq \mu(w^i_\alpha)-\mu(U^i\setminus w^i_\beta)
>\mu(w^i_\alpha) - \frac{\delta-(1-\varepsilon)}{n}
\end{align*}
In particular, this gives that \(\sum_i\mu(w^i_\alpha\cap w^i_\beta)>1-\varepsilon\).

Since \(\Q\) is Knaster we may assume that the elements of
\(\set{\tau_\alpha^i}{\alpha<\omega_1}\) are pairwise compatible and
that this holds for any \(i\). Pick lower bounds \(q^i_{\alpha\beta}\)
for the \(\tau_\alpha^i\) and \(\tau_\beta^i\). By mixing we can
construct \(\Bnull\)-names \(\sigma_{\alpha\beta}\) for elements of \(\Q\) such that  
\(w_\alpha^i\cap w_\beta^i\forces \sigma_{\alpha\beta}=q^i_{\alpha\beta}\)
for all \(i\). By construction the \(\sigma_{\alpha\beta}\) are 
\(\varepsilon\)-deficient finite \(\Bnull\)-mixtures and are lower bounds for
\(\tau_\alpha\) and \(\tau_\beta\).
\end{proof}

While the concept of \(\varepsilon\)-deficient finite mixtures makes sense for any measure
algebra, finding a good analogue of the preceding proposition for algebras of
uncountable weight has proven difficult.

\begin{lemma}
\label{lemma:MeasureOfSupOfDirectedInfs}
Let \((\B,m)\) be a measure algebra and \(0<\varepsilon<1\).
Suppose \(\mathcal{A}\) is a family of finite antichains in \(\B\), downward
directed under refinement, such that \(m(\sup A)>1-\varepsilon\) for any 
\(A\in\mathcal{A}\). If we let \(d_{\mathcal{A}}=\inf\set{\sup A}{A\in\mathcal{A}}\)
then \(m(d_{\mathcal{A}})\geq 1-\varepsilon\).
\end{lemma}

\begin{proof}
By passing to complements it suffices to prove the following statement: if \(I\)
is an upward directed subset of \(\B\) all of whose elements have measure less than
\(\varepsilon\) then \(\sup I\) has measure at most \(\varepsilon\).

Using the fact that \(\B\) is complete, we can refine \(I\) to an antichain \(Z\) that
satisfies \(\sup I=\sup Z\). Since \(\B\) is ccc, \(Z\) must be countable. Applying the
upward directedness of \(I\) and the countable additivity of the measure, we can conclude
that \(m(\sup Z)\leq \varepsilon\).\qedhere
\end{proof}

We are finally ready to state and prove the preservation theorem we have been building
towards.

\begin{theorem}
\label{thm:RandomGivesgrMA}
Assume Martin's Axiom holds in \(V\) and let \(V[r]\) be obtained by adding a random 
real to \(V\). Then \(V[r]\) satisfies the grounded Martin's Axiom over the ground 
model \(V\).
\end{theorem}

\begin{proof}
If \(\mathrm{CH}\) holds in \(V\) then it holds in \(V[r]\) as well, implying that
\(V[r]\) satisfies the full Martin's Axiom. We may therefore assume without loss of 
generality that \(V\) satisfies \(\ma+\lnot\mathrm{CH}\).

Assume toward a contradiction that \(V[r]\) does not satisfy the grounded Martin's
Axiom over \(V\). Then there exist
a poset \(\Q\in V\) which is ccc in \(V[r]\), a cardinal \(\kappa<\c\) and a collection
\(\mathcal{D}=\set{D_\alpha}{\alpha<\kappa}\in V[r]\) of dense subsets of \(\Q\) such that
\(V[r]\) has no \(\mathcal{D}\)-generic filters on \(\Q\). There must be a condition
\(b_0\in\Bnull\) forcing this. Let \(\varepsilon<\mu(b_0)\).

Since \(\Bnull\) is ccc, \(\Q\) must be ccc in \(V\) and, since \(\ma+\lnot\mathrm{CH}\) 
holds there, is also Knaster there. 
Thus \(\Termfin^\varepsilon(\Bnull,\Q)\in V\) is Knaster by 
lemma~\ref{lemma:RandomDeficientMixKnaster}. We now choose names 
\(\dot{D}_\alpha\)
for the dense sets \(D_\alpha\) such that \(\Bnull\) forces that the \(\dot{D}_\alpha\) are dense
subsets of \(\Q\). Then, by lemma~\ref{lemma:DeficientFiniteMixDenseTranslation}, the
\(E_\alpha\) are dense in \(\Termfin^\varepsilon(\Bnull,\Q)\), where \(E_\alpha\) is defined
from \(\dot{D}_\alpha\) as in that lemma. We can thus obtain, using Martin's Axiom 
in \(V\), a filter
\(H\) on \(\Termfin^\varepsilon(\Bnull,\Q)\) which meets all of the \(E_\alpha\).

Pick a resolving antichain \(A_\tau\) for each \(\tau\in H\) and consider 
\(\mathcal{A}=\set{A_\tau}{\tau\in H}\). This family satisfies the hypotheses of
lemma~\ref{lemma:MeasureOfSupOfDirectedInfs}, whence we can conclude that
\(\mu(d_{\mathcal{A}})\geq 1-\varepsilon\), where \(d_{\mathcal{A}}\) is defined as in
that lemma.
Interpreting \(H\) as a \(\Bnull\)-name for a subset of \(\Q\), we now observe that 
\[d_{\mathcal{A}}\forces_{\Bnull}
\text{``\(H\) generates a \(\dot{\mathcal{D}}\)-generic filter on \(\Q\)''}\]
Now, crucially, since we have chosen \(\varepsilon<\mu(b_0)\), the conditions \(b_0\) and
\(d_H\) must be compatible in \(\Bnull\). But this is a contradiction, since they
force opposing statements. Therefore \(V[r]\) really does satisfy
the grounded Martin's Axiom over \(V\).
\end{proof}

\begin{corollary}
\label{cor:grMAConsistentWithNoSuslinTrees}
The grounded Martin's Axiom is consistent with there being no Suslin trees.
\end{corollary}

\begin{proof}
If Martin's Axiom holds in \(V\) and \(r\) is random over \(V\) then \(V[r]\) satisfies
the grounded Martin's Axiom by the above theorem and also has no Suslin trees by a 
theorem of Laver~\cite{Laver1987:RandomRealsSuslinTrees}.
\end{proof}

Unfortunately, the employed techniques do not seem to yield the full preservation result 
as in theorem~\ref{thm:CohenPreservesgrMA}. If \(V\) satisfied merely the grounded
Martin's Axiom over a ground model \(W\) we would have to argue that the
poset \(\Termfin^\varepsilon(\Bnull,\Q)\) as computed in \(V\) was actually an element of \(W\), so
that we could apply \grma to it. But we cannot expect this to be true if passing
from \(W\) to \(V\) added reals; not only will the termspace posets be computed
differently in \(W\) and in \(V\), even the random Boolean algebras of these two
models will be different. Still, these considerations lead us to the following
improvement to the theorem above.

\begin{theorem}
\label{thm:RandomPreservesDistributivegrMA}
Assume the grounded Martin's Axiom holds in \(V\) over the ground model \(W\)
via a forcing which is countably distributive (or, equivalently, does not add reals),
and let \(V[r]\) be obtained by adding a random real to \(V\). Then \(V[r]\) also
satisfies the grounded Martin's Axiom over the ground model \(W\).
\end{theorem}

\begin{proof}
By assumption there is a ccc countably distributive poset \(\P\in W\) such that
\(V=W[G]\) for some \(W\)-generic \(G\subseteq \P\). Since \(W\) and \(V\) thus have
the same reals, they must also have the same Borel sets. Furthermore, since the
measure is inner regular, a Borel set having positive measure is witnessed by a
positive measure compact (i.e.\ closed) subset, which means that \(W\) and \(V\)
agree on which Borel sets are null. It follows that the random Boolean algebras
as computed in \(W\) and in \(V\) are the same.

Now let \(0<\varepsilon<1\) and let \(\Q\in W\) be a poset which is ccc in \(V\).
We claim that \(V\) and \(W\) compute the poset \(\Termfin^\varepsilon(\Bnull,\Q)\)
the same. Clearly any \(\varepsilon\)-deficient finite mixture in \(W\)
is also such in \(V\), so we really only need to see that \(V\) has no new such
elements. But \(\Bnull\) is ccc, which means that elements of \(\Q\) have countable
nice names and these could not have been added by \(G\). So \(V\) and \(W\) in fact
agree on the whole termspace poset \(\Term(\Bnull,\Q)\), and therefore also
on the \(\varepsilon\)-deficient finite mixtures.

The rest of the proof proceeds as in theorem~\ref{thm:RandomGivesgrMA}. The key
step there, where we apply Martin's Axiom to the poset 
\(\Termfin^\varepsilon(\Bnull,\Q)\), goes through, since we have shown that
this poset is in \(W\) and we may therefore apply \grma to it.
\end{proof}

Just as in theorem~\ref{thm:CohenPreservesgrMA} we may replace the grounded Martin's
Axiom in the above theorem with its local version.

It is not immediately obvious that the hypothesis of the above theorem is
ever satisfied in a nontrivial way, that is, whether \grma can ever hold via a
\emph{nontrivial} countably distributive extension. The following theorem,
due to Larson, shows
that this does happen and gives yet another construction of a model of \grma.
For the purposes of this theorem we shall call a Suslin tree \(T\) \emph{homogeneous}
if for any two nodes \(p,q\in T\) of the same height, the cones below them are
isomorphic. Note that homogeneous Suslin trees may be constructed from \(\diamondsuit\).

\begin{theorem}[Larson~\cite{Larson1999:SmaxVariationForOneSuslinTree}]
\label{thm:grMAAfterSuslinTree}
Let \(\kappa>\omega_1\) be a regular cardinal satisfying \(\kappa^{<\kappa}=\kappa\)
and let \(T\) be a homogeneous Suslin tree. Then there is a ccc poset \(\P\) such that, given a
\(V\)-generic \(G\subseteq\P\), the tree \(T\) remains Suslin in \(V[G]\) and, if
\(b\) is a generic branch through \(T\), the extension \(V[G][b]\) satisfies
\(\c=\kappa\) and the grounded Martin's Axiom over the ground model \(V[G]\).
\end{theorem}

\begin{proof}
The idea is to attempt to force \(\ma +\c=\kappa\), but only using posets that
preserve the Suslin tree \(T\). More precisely, fix a well-order \(\triangleleft\)
of \(H_\kappa\) of length \(\kappa\) and define \(\P\) as the length \(\kappa\) finite
support iteration which forces at stage \(\alpha\) with the next \(\P_\alpha\)-name
for a poset \(\dot{\Q}_\alpha\) such that \(\P_\alpha\) forces that 
\(T\times\dot{\Q}_\alpha\) is ccc.

Let \(G\subseteq\P\) be \(V\)-generic. It is easy to see by induction
that \(\P_\alpha\times T\) is ccc for all \(\alpha\leq\kappa\); 
the successor case is clear from the definition of the iteration \(\P\)
and the limit case follows by a \(\Delta\)-system argument. We can thus conclude that
\(T\) remains a Suslin tree in \(V[G]\). Furthermore, standard arguments show that
there are exactly \(\kappa\) many reals in \(V[G]\) and that this extension satisfies
Martin's Axiom for small posets which preserve \(T\), i.e.\ those \(\Q\) such that
\(\Q\in H_\kappa^{V[G]}\) and \(\Q\times T\) is ccc.

Finally, let us see that adding a branch
\(b\) through \(T\) over \(V[G]\) yields a model of the grounded Martin's Axiom over
\(V[G]\). Thus let \(\Q\in V[G]\) be a poset which is ccc in \(V[G][b]\) and has size
less than \(\kappa\) there. There is a condition in \(T\) forcing that \(\Q\) is ccc,
so by our homogeneity assumption \(T\) forces this, meaning that \(\Q\times T\) is ccc
in \(V[G]\).
The key point now is that, since \(T\) is countably
distributive, all of the maximal antichains (and open dense subsets) of \(\Q\) in 
\(V[G][b]\) are already in \(V[G]\). Furthermore, any collection \(\mathcal{D}\) 
of less than \(\kappa\)
many of these in \(V[G][b]\) can be covered by some \(\widetilde{\mathcal{D}}\)
in \(V[G]\) of the same size. Our observation from the previous paragraph then
yields a \(\widetilde{\mathcal{D}}\)-generic filter for \(\Q\) in \(V[G]\)
and therefore \(V[G][b]\) satisfies \grma over \(V[G]\) by lemma~\ref{lemma:grMASmallPosets}.
\end{proof}

Starting from a Suslin tree with a stronger homogeneity property, Larson
also shows that there are no Suslin trees in the extension \(V[G][b]\) above.
This gives an alternative proof of corollary~\ref{cor:grMAConsistentWithNoSuslinTrees}.

From the argument of theorem~\ref{thm:grMAAfterSuslinTree} we can actually extract 
another preservation result for \grma.

\begin{theorem}
\label{thm:SuslinPreservesgrMA}
Assume the grounded Martin's Axiom holds in \(V\) over the ground model \(W\)
and let \(T\in V\) be a Suslin tree. If \(b\subseteq T\) is a generic branch
then \(V[b]\) also satisfies the grounded Martin's Axiom over the ground model \(W\).
\end{theorem}

\begin{proof}
The point is that, just as in the proof of theorem~\ref{thm:grMAAfterSuslinTree},
forcing with \(T\) does not add any new maximal antichains to posets from \(W\)
that remain ccc in \(V[b]\) and any collection of these antichains in \(V[b]\) can 
be covered by a collection of the same size in \(V\).
\end{proof}

Starting from a Suslin tree with a stronger homogeneity property, Larson also shows
that there are no Suslin trees at all in the extension \(V[G][b]\) above. This
shows that \(\grma+\lnot\ma+\text{``there are no Suslin trees''}\) is consistent
(although this is also true in our model from theorem~\ref{thm:RandomGivesgrMA} by
a result of Laver~\cite{Laver1987:RandomRealsSuslinTrees}). On the other hand,
our models from theorems~\ref{thm:CanonicalgrMA} and~\ref{thm:CohenPreservesgrMA}
show the consistency of \(\grma+\text{``there is a Suslin tree''}\) by a result of
Shelah~\cite{Shelah1984:TakeSolovayInaccessibleAway}. 

If \grma holds over a ground model that reals have been added to, it seems harder
to say anything about preservation after adding a further random real. Nevertheless,
we fully expect the answers to the following questions to be positive.

\begin{question}
Does adding a random real to a model of \(\grma\) preserve \(\grma\)? Does it preserve
it with the same witnessing ground model?
\end{question}

Generalizations of theorems~\ref{thm:CohenPreservesgrMA} and~\ref{thm:RandomGivesgrMA}
to larger numbers of reals added seem the natural next step in the exploration of
the preservation phenomena of the grounded Martin's Axiom. 
Such preservation results would also help in determining the compatibility
of \grma with various configurations of the cardinal characteristics on the left
side of Cichoń's diagram.
The constructions \(\Termfin\) and \(\Termfin^\varepsilon\)
seem promising, but obtaining a good chain condition in any case at all, except those shown,
has proven difficult.

\section{The grounded Proper Forcing Axiom}

We can, of course, also consider grounded versions of other forcing axioms. We
define one and note that similar definitions can be made for
\(\mathrm{grSPFA},\mathrm{grMM}\) and so on.

\begin{definition}
The \emph{grounded Proper Forcing Axiom} (\grpfa) asserts that \(V\) is a forcing
extension of some ground model \(W\) by a proper poset and \(V\) satisfies the
conclusion of the Proper Forcing Axiom for posets \(\Q\in W\) which are still
proper in \(V\).
\end{definition}

\begin{theorem}
\label{thm:CanonicalgrPFA}
Let \(\kappa\) be supercompact. Then there is a proper forcing extension that
satisfies \(\grpfa+\c=\kappa=\omega_2\) and in which \pfa, and even \ma, fails.
\end{theorem}

\begin{proof}
Start with a Laver function \(\ell\) for \(\kappa\) and build a countable-support 
forcing iteration \(\P\) of length \(\kappa\) which forces at stage \(\alpha\)
with \(\dot{\Q}\), some full name for the poset \(\Q=\ell(\alpha)\) if it is proper at 
that stage and with trivial forcing otherwise. Note that \(\P\) is proper. 
Now let \(V[G][H]\) be a forcing extension by \(\R=\Add(\omega,1)\times\P\).
We claim that \(V[G][H]\) is the required model.

Since the Laver function \(\ell\) will quite often output the poset
\(\Add(\omega,1)\) and this will always be proper, the iteration \(\P\) will 
add reals unboundedly often. 
Furthermore, since \(\R\) is \(\kappa\)-cc, we will obtain \(\c=\kappa\) in \(V[G][H]\).

Next we wish to see that \(\kappa=\omega_2^{V[G][H]}\). For this it suffices to see
that any \(\omega_1<\lambda<\kappa\) is collapsed at some point during the iteration.
Recall the well-known fact that any countable-support iteration of nontrivial posets 
adds a Cohen subset of \(\omega_1\) at stages of cofinality \(\omega_1\) and therefore
collapses the continuum to \(\omega_1\) at those stages. Now fix some \(\omega_1<\lambda
<\kappa\). Since \(\P\) ultimately adds \(\kappa\) many reals and is \(\kappa\)-cc,
there is some stage \(\alpha\) of the iteration such that \(\P_\alpha\) has already
added \(\lambda\) many reals and therefore \(\c^{V[H_\alpha]}\geq\lambda\).
Since \(\P_\alpha\) is proper, the rest of the iteration \(\P\rest[\alpha,\kappa)\)
is a countable-support iteration in \(V[H_\alpha]\) and the fact mentioned above
implies that \(\lambda\) is collapsed to \(\omega_1\) by this tail of the iteration.

Note that \(\R\) is proper, since \(\P*\Add(\omega,1)\) is proper and has a dense subset
isomorphic to \(\R\). 
To verify that \grpfa holds in \(V[G][H]\) let \(\Q\in V\) be a poset that is
proper in \(V[G][H]\) and let \(\mathcal{D}=\set{D_\alpha}{\alpha<\omega_1}\in V[G][H]\)
be a family of dense subsets of \(\Q\). In \(V\) we can fix (for some large enough 
\(\theta\)) a \(\theta\)-supercompactness embedding \(j\colon V\to M\) such that
\(j(\ell)(\kappa)=\Q\). Since the Cohen real forcing is small, the embedding \(j\)
lifts to a \(\theta\)-supercompactness embedding \(j\colon V[G]\to M[G]\).
We can factor \(j(\P)\) in \(M[G]\) as \(j(\P)=\P*\Q*\Ptail\). Let
\(h*\Htail\subseteq \Q*\Ptail\) be \(V[G][H]\)-generic. As usual, we can now lift
the embedding \(j\) in \(V[G][H*h*\Htail]\) to 
\(j\colon V[G][H]\to \overline{M}=M[G][H*h*\Htail]\).
Note that the closure of this embedding implies that \(j[h]\in \overline{M}\).
But \(j[h]\) is a \(j(\mathcal{D})\)-generic filter on \(j(\Q)\) in
\(\overline{M}\) and so, by elementarity, there is a \(\mathcal{D}\)-generic filter
on \(\Q\) in \(V[G][H]\).

Finally, since we can see \(V[G][H]\) as obtained by adding a Cohen real to an 
intermediate extension and since CH fails there, PFA and even Martin's Axiom must fail 
there by Roitman's~\cite{Roitman1979:AddingRandomOrCohenRealEffectMA}.
\end{proof}

With regard to the above proof, we should mention that one usually argues that
\(\kappa\) becomes \(\omega_2\) after an iteration similar to ours because
at many stages the poset forced with was explicitly a collapse poset
\(\Coll(\omega_1,\lambda)\).
In our case, however, the situation is different. It turns out that
a significant number of proper posets from \(V\) (the collapse posets among them)
cease to be proper as soon as we add the initial Cohen real. 
Therefore the possibility of choosing
\(\Coll(\omega_1,\lambda)\) never arises in the construction of the iteration
\(\P\) and a different argument is needed. We recount a proof of this
fact below. The argument is essentially due to Shelah, as communicated
by Goldstern in~\cite{MO:Goldstern2015:PreservationOfProperness}.

\begin{theorem}[Shelah]
\label{thm:NotProperAfterAddingReal}
Let \(\P\) be a ccc poset and let \(\Q\) be a countably distributive poset which
collapses \(\omega_2\). Let \(G\subseteq\P\) be \(V\)-generic. If \(V[G]\)
has a new real then \(\Q\) is not proper in \(V[G]\).
\end{theorem}

\begin{proof}
Fix at the beginning a \(\Q\)-name \(\dot{f}\), forced to be a bijection between 
\(\omega_1\) and \(\omega_2^V\). Let \(\theta\) be a sufficiently large regular cardinal.
By claim XV.2.12 of~\cite{Shelah1998:ProperImproperForcing} we can label the nodes
\(s\in\funcs{<\omega}{2}\) with countable models \(M_s\prec H_\theta\) such that:
\begin{itemize}
\item the \(M_s\) are increasing along each branch of the tree \(\funcs{<\omega}{2}\);
\item \(\P,\Q,\dot{f}\in M_\emptyset\);
\item there is an ordinal \(\delta\) such that \(M_s\cap\omega_1=\delta\) for all \(s\);
\item for any \(s\) there are ordinals \(\alpha_s<\beta_s<\omega_2\) such that
\(\alpha_s\in M_{s\concat 0}\) and \(\alpha_s\notin M_{s\concat 1}\) and 
\(\beta_s \in M_{s\concat 1}\).
\end{itemize}

Now consider, in \(V[G]\), the tree of models \(M_s[G]\). By the argument given in
the proof of lemma~\ref{lemma:grMASmallPosets}, the models \(M_s[G]\) are elementary
in \(H_\theta^{V[G]}\) and, since \(\P\) is ccc, we still have
\(M_s[G]\cap\omega_1=\delta\). Let \(M=M_r[G]\) be the branch model determined by
the new real \(r\in V[G]\). We shall show that there are no \(M\)-generic
conditions in \(\Q\).

Suppose that \(q\) were such a generic condition. We claim that \(q\) forces that
\(\dot{f}\rest\delta\) maps onto \(M\cap\omega_2^V\) (note that \(\dot{f}\) still names
a bijection \(\omega_1\to\omega_2^V\) over \(V[G]\)). First, suppose that \(q\) does not
force that \(\dot{f}[\delta]\subseteq M\). Then we can find \(q'\leq q\) and an
\(\alpha<\delta\) such that \(q'\forces\dot{f}(\alpha)\notin M\).
But if \(q'\in H\subseteq\Q\) is generic then \(M[H]\) is an elementary substructure of
\(H_\theta^{V[G][H]}\) and, of course, \(f,\alpha\in M[H]\), leading to a contradiction.

Conversely, suppose that \(q\) does not force that
\(M\cap\omega_2^V\subseteq \dot{f}[\delta]\). We can again find \(q'\leq q\) and an
\(\alpha\in M\cap\omega_2\) such that \(q'\forces\alpha\notin\dot{f}[\delta]\).
Let \(q'\in H\subseteq\Q\) be generic. As before, \(M[H]\) is an elementary substructure
of \(H_\theta^{V[G][H]}\) and \(f,\alpha\in M[H]\). Since 
\(f\colon\omega_1\to\omega_2^V\) is a bijection, we must have \(f^{-1}(\alpha)\in M[H]\).
But by construction \(f^{-1}(\alpha)\) is an ordinal greater than \(\delta\)
while simultaneously \(M[H]\cap\omega_1=\delta\) by the \(M\)-genericity of \(q\),
giving a contradiction.

Fixing our putative generic condition \(q\), we can use the countable distributivity of 
\(\Q\) in \(V\) to see that \(\dot{f}\rest\delta\), and consequently \(M\cap\omega_2\),
exist already in \(V\). But we can extract \(r\) from \(M\cap\omega_2\).

Notice that, given a model \(M_{s\concat 1}\) in our original tree, no elementary
extension \(M_{s\concat 1}\prec X\) satisfying \(X\cap\omega_1=\delta\)
can contain \(\alpha_s\). This is because \(M_{s\concat 1}\) contains a bijection
\(g\colon \omega_1\to\beta_s\) and, by elementarity, \(g\) must restrict to a bijection
between \(\delta\) and \(M_{s\concat 1}\cap \beta_s\). 
But seen from the viewpoint of \(X\), that same function \(g\) must restrict to a
bijection between \(\delta\) and \(X\cap \beta_s\) and so \(X\cap \beta_s=
M_{s\concat 1}\cap\beta_s\).

Using this fact, we can now extract \(r\) from \(M\cap\omega_2\) in \(V\). Specifically,
we can decide at each stage whether the branch determined by \(r\) went left or
right depending on whether \(\alpha_s\in M\cap\omega_2\) or not. We conclude that
\(r\) appears already in \(V\), contradicting our original assumption. Therefore
there is no generic condition \(q\) as above and \(\Q\) is not proper in \(V[G]\).
\end{proof}

%For example, if \(\R=\Add(\omega,1)\times\P\) is as in the proof of 
%theorem~\ref{thm:CanonicalgrPFA} then theorem~\ref{thm:NotProperAfterAddingReal}
%shows that the posets \(\Coll(\omega_1,\lambda)^V\) for \(\lambda\geq\omega_2\)
%cease to be proper as soon as the first Cohen real is added and are thus never forced
%with in the iteration \(\P\).

Ultimately, one hopes that by grounding the forcing axiom we lower its consistency 
strength while still being able to carry out at least some of the usual arguments and 
obtain at least some of the standard consequences. However, 
theorem~\ref{thm:NotProperAfterAddingReal} severely limits the kind of arguments we
can carry out under \grpfa. Many arguments involving \pfa use, among other things,
collapsing posets such as \(\Coll(\omega_1,2^\omega)\).
In contrast, if the poset witnessing \grpfa in a model is
any kind of iteration that at some stage added, say, a Cohen real, the theorem
prevents us from applying the forcing axiom to any of these collapsing posets.
It is thus unclear exactly how much strength of \pfa can be recovered from \grpfa.
In particular, while \grpfa implies that CH fails, the following key question 
remains open:

\begin{question}
\label{q:grPFAContinuum}
Does \grpfa imply that the continuum equals \(\omega_2\)?
\end{question}

Regarding the relation of \grpfa to other forcing axioms, a lot remains unknown.
Theorem~\ref{thm:CanonicalgrPFA} shows that \grpfa does not imply \ma. Beyond this
a few more things can be said.

\begin{proposition}
Martin's Axiom does not imply \grpfa.
\end{proposition}

\begin{proof}
Starting over some model, force with the Solovay-Tennenbaum iteration to produce
a model \(V[G]\) satisfying \ma and \(\c=\omega_3\). Now perform Reitz's ground
axiom forcing (cf.~\cite{Reitz2007:GroundAxiom}) above \(\omega_3\) to produce a model
\(V[G][H]\) still satisfying \ma and \(\c=\omega_3\) but which is not a
set-forcing extension of any model (note that \(H\) is added by class-sized forcing).
Therefore the only way \(V[G][H]\) could satisfy \grpfa is if it actually
satisfied \pfa in full. But that cannot be the case since \pfa implies that
the continuum equals \(\omega_2\).
\end{proof}

\begin{proposition}
The grounded Proper Forcing Axiom does not imply \mas (and not even \(\mathrm{MA}_{\aleph_1}(\textup{\(\sigma\)-centred})\)).
\end{proposition}

\begin{proof}
We could have replaced the forcing \(\Add(\omega,1)\times\P\)
in the proof of theorem~\ref{thm:CanonicalgrPFA} with \(\Add(\omega,\omega_1)\times\P\)
without issue. 
As in the proof of theorem~\ref{thm:grMASmallCharacts} we get a model
whose bounding number equals \(\aleph_1\), but this contradicts
\(\mathrm{MA}_{\aleph_1}(\text{\(\sigma\)-centred})\) as in the proof of
corollary~\ref{cor:grMANotImpliesMAsigma}.
\end{proof}

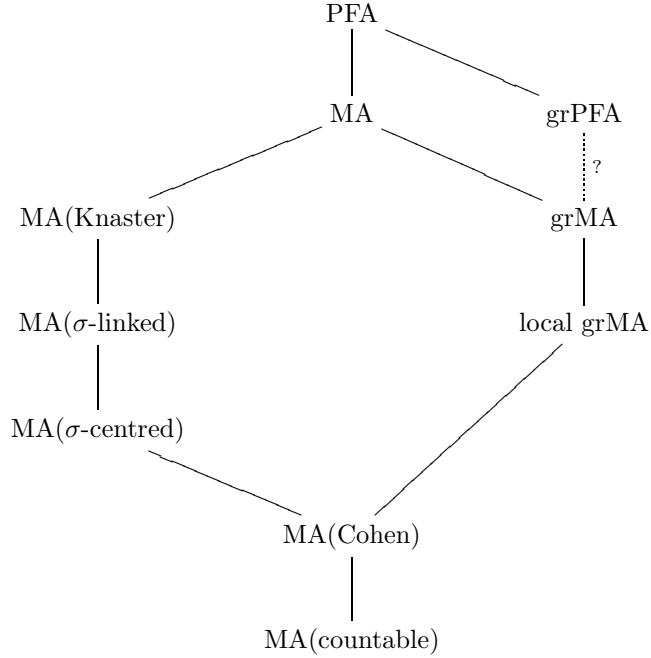
\begin{figure}[h!t]
\[
\xymatrix{
&\pfa\ar@{-}[d]\ar@{-}[dr]\\
&\ma\ar@{-}[dl]\ar@{-}[dr]& \grpfa\ar@{.}[d]^{?}\\
%&\ma\ar@{-}[dl]\ar@{-}[dr]&& \grpfa\ar@{.}[dl]|-{???}\ar@/^2.5pc/@{.}[dddddll]|-{???}\\
\mak\ar@{-}[d]&&\grma\ar@{-}[d]\\
\textup{MA(\(\sigma\)-linked)}\ar@{-}[d]&&\text{local \grma}\ar@{-}[ddl]\\
\mas\ar@{-}[dr]\\
&\mathrm{MA(Cohen)}\ar@{-}[d]\\
&\mathrm{MA(countable)}}
\]
\end{figure}

While it is easy to see that \grpfa implies \(\mathrm{MA}_{\aleph_1}(\mathrm{Cohen})\),
whether or not it even implies \mac is unclear (a large part of the problem being that
we do not have an answer to question~\ref{q:grPFAContinuum}). But an even more
pressing question concerns the relationship between \grpfa and \grma:

\begin{question}
Does \grpfa imply \grma?
\end{question}

Even if the answer to question~\ref{q:grPFAContinuum} turns out to be positive, we
conjecture that the answer to this last question is negative. We expect that
it is possible to use the methods of \cite{Reitz2007:GroundAxiom} or, more
generally, \cite{FuchsHamkinsReitz2015:SetTheoreticGeology} in combination
with the forcing construction of theorem~\ref{thm:CanonicalgrPFA} to produce
a model of \grpfa which has no ccc ground models and in which \ma (and consequently
also \grma) fails.

\bibliographystyle{amsplain}
\bibliography{bibbase}
\end{document}